\documentclass[preprint]{article}

\usepackage{graphicx}
\usepackage{amsmath}
\usepackage{amsthm}
\usepackage{amssymb}
\usepackage{bm}
\usepackage{array}
\usepackage{authblk}

\usepackage{epstopdf}
\newtheorem{theorem}{Theorem}
\newtheorem{lemma}{Lemma}
\newtheorem{remark}{Remark}
\newtheorem{condition}{Condition}

\begin{document}

\title{An SDE approximation for stochastic differential delay equations with state-dependent colored noise}
\date{}

\author[1]{Austin McDaniel}
\author[2]{\"Ozer Duman}
\author[2,3]{Giovanni Volpe}
\author[1]{Jan Wehr}
\affil[1]{Department of Mathematics, University of Arizona, Tucson, Arizona 85721 USA}
\affil[2]{Soft Matter Lab, Department of Physics, Bilkent University, Cankaya, Ankara 06800, Turkey}
\affil[3]{UNAM - National Nanotechnology Research Center, Bilkent University, Cankaya, Ankara 06800, Turkey}

\maketitle

\begin{abstract}

\textbf{
We consider a general multidimensional stochastic differential delay equation (SDDE) with state-dependent colored noises. We approximate it by a stochastic differential equation (SDE) system and calculate its limit as the time delays and the correlation times of the noises go to zero. The main result is proven using a theorem about convergence of stochastic integrals by Kurtz and Protter. It formalizes and extends a result that has been obtained in the analysis of a noisy electrical circuit with delayed state-dependent noise, and may be used as a working SDE approximation of an SDDE modeling a real system where noises are correlated in time and whose response to noise sources depends on the system's state at a previous time.
} \\

\noindent \emph{Keywords}: Stochastic differential equations, stochastic differential delay equations, colored noise, noise-induced drift \\

\noindent AMS Subject Classification: 60H10, 34F05

\end{abstract}

\section{Introduction}  

\indent 

Stochastic differential equations (SDEs) are widely employed to describe the time evolution of systems encountered in physics, biology, and economics, among others \cite{Oksendal, Arnold, Karatzas}. It is often natural to introduce a delay into the equations in order to account for the fact that the system's response to changes in its environment is not instantaneous. We are, therefore, led to consider stochastic differential delay equations (SDDEs).  A survey of the theory of SDDEs, including theorems on existence and uniqueness of solutions as well as stochastic stability, can be found in Ref.~\cite{Ivanov}.  In addition to numerous other results, a treatment of the (appropriately defined) Markov property and the concept of a generator are contained in  Ref.~\cite{Kushner1}.  Numerical aspects of SDDEs are treated in  Ref.~\cite{Kushner2}.  For other aspects of the theory see Ref.~\cite{Mao}.  

Since the theory of SDDEs is much less developed than the theory of SDEs \cite{Oksendal, Arnold, Karatzas}, it is useful to introduce working approximations of SDDEs by SDEs. For example, such an approximation was applied in Ref.~\cite{Pesce} to a physical system with one dynamical degree of freedom (the output voltage of a noisy electrical circuit).  It was used there to show that the experimental system shifts from obeying Stratonovich calculus to obeying It\^o calculus as the ratio between the driving noise correlation time and the feedback delay time changes (see \cite{Polettini} for related work). In this article we employ the  systematic and rigorous method developed in Ref.~\cite{Hottovy} to obtain much more general results which are applicable to systems with an arbitrary number of degrees of freedom, driven by several colored noises, and involving several time delays. More precisely, we derive an approximation of SDDEs driven by colored noise (or noises) in the limit in which the correlation times of the noises and the response delays go to zero at the same rate.   The approximating equation contains {\it noise-induced drift} terms which depend on the ratios of the delay times to the noise correlation times.

An equation related to, but simpler than, the one considered here was studied in a different context in  Ref.~\cite{Pavliotis}.  There, the limit that the authors derive is analogous to our Theorem 1.  Results on small delay approximations for SDDEs of a different type than the one considered here are contained in Ref.~\cite{Guillouzic}; see also Ref.~\cite{Longtin}.  We are not aware of any previous studies addressing the question of the effective equation in the limit as the time delays and correlation times of the noises go to zero, other than a less mathematical and less general treatment in our previous work~\cite{Pesce}.  In fact, the present paper was motivated by~\cite{Pesce} and can be seen as its mathematically formal extension.

\section{Mathematical Model}

We consider the multidimensional SDDE system
\begin{equation} \label{SDDE}
d \bm{x} _t = \bm{f}(\bm{x} _t) dt + \bm{g}(\bm{x}_{t - \delta}) \bm{\eta} _t dt
\end{equation}
where $\bm{x} _t = (x^1 _t, ..., x^i _t, ..., x^m _t)^{\rm T}$ is the state vector (the superscript ${\rm T}$ denotes transposition), $\bm{f}(\bm{x} _t) = (f^1 (\bm{x}_t), ..., f^i (\bm{x}_t), ..., f^m (\bm{x}_t))^{\rm T}$ where $\bm{f}$ is a vector-valued function describing the deterministic part of the dynamical system, 
$$\bm{g}(\bm{x}_{t - \delta}) =
\left[\begin{array}{ccccc}
g^{11}(\bm{x}_{t - \delta}) & \dots & g^{1j}(\bm{x}_{t - \delta}) & \dots & g^{1n}(\bm{x}_{t - \delta})  \\
\vdots & \ddots & \vdots & \ddots & \vdots \\
g^{i1}(\bm{x}_{t - \delta}) & \dots & g^{ij}(\bm{x}_{t - \delta}) & \dots & g^{in}(\bm{x}_{t - \delta}) \\
\vdots & \ddots & \vdots & \ddots & \vdots \\
g^{m1}(\bm{x}_{t - \delta}) & \dots &  g^{mj}(\bm{x}_{t - \delta}) & \dots & g^{mn}(\bm{x}_{t - \delta})
\end{array}\right]
$$
where $\bm{g}$ is a matrix-valued function, $\bm{x} _{t - \delta} = (x^1 _{t - \delta _1}, ..., x^i _{t - \delta _i}, ..., x^m _{t - \delta _m})^{\rm T}$ is the delayed state vector (note that each component is delayed by a possibly different amount $\delta _i > 0$), and $\bm{\eta} _t = (\eta ^1 _t, ..., \eta ^j _t, ..., \eta ^n _t)^{\rm T}$ is a vector of independent noises $\eta^j$, where the $\eta^j$ are colored (harmonic) noises with characteristic correlation times $\tau_j$.  These stochastic processes (defined precisely in equation~\eqref{eq:harmonicnoise}) have continuously differentiable realizations which makes the realizations of the solution process $\bm{x}$ twice continuously differentiable under the natural assumptions on ${\bm f}$ and ${\bm g}$ that are made in the statement of Theorem 1.

Equation~\eqref{SDDE} is written componentwise as
\begin{equation} \label{SDDE component}
\frac{dx^i(t)}{ dt} = f^i(x^1(t), \; \dots \; , x^m(t)) + \sum_{j=1}^n g^{ij}(x^1(t-\delta_1), \; \dots \;, x^m(t - \delta_m)) \eta^j(t) \; .
\end{equation}
For each $i$, we define the process $y^i(t) = x^i(t - \delta_i)$.  In terms of the $y$ variables, equation~\eqref{SDDE component}  becomes
\begin{equation} \label{SDDE component shifted}
\frac{dy^i(t + \delta _i)}{dt} = f^i(y^1(t + \delta _1), \; \dots \; , y^m(t + \delta _m)) + \sum_{j=1}^n g^{ij}(y^1(t), \; \dots \; , y^m(t)) \eta^j(t) \; .
\end{equation}
Expanding to first order in $\delta_i$, we have $\dot{y}^i (t+\delta _i) \cong \dot{y}^i (t) + \delta _i \ddot{y}^i (t)$ and 
\begin{align*}
f^i(y^1(t + \delta _1), \; \dots \; , y^m(t + \delta _m)) & \cong f^i(y^1(t), \; \dots \; , y^m(t)) \\
& + \sum _{k = 1}^m \delta _k \frac{\partial f^i (y^1(t), \; \dots \; , y^m(t))}{\partial y_k} \frac{d y^k (t)}{dt} \; .
\end{align*}
Substituting these approximations into equation~\eqref{SDDE component shifted}, we obtain a new (approximate) system
$$
\frac{dy^i(t)}{dt} + \delta _i \frac{d^2 y^i(t)}{dt ^2}= f^i(\textbf{y}(t)) + \sum _{k = 1}^m \delta _k \frac{\partial f^i (\textbf{y}(t))}{\partial y_k} \frac{d y^k (t)}{dt} + \sum_{j=1}^n g^{ij}(\textbf{y}(t)) \eta^j(t)
$$
where $\textbf{y}(t) = (y^1(t), \; \dots \; , y^m(t))^{\rm T}$.
We write these equations as the first order system
\begin{equation}\label{MainSDE}
\left\{
\begin{array}{ccl}
\displaystyle dy^i _t & = &  \displaystyle v ^i _t dt \\[12pt]
\displaystyle dv^i _t & = & \displaystyle \left[- \frac{1}{\delta _i} v^i _t + \frac{1}{\delta _i} f^i (\textbf{y}_t)  +  \frac{1}{\delta _i} \sum_{k =1}^m  \delta _k \frac{\partial f ^i (\textbf{y}_t)}{\partial y _k} v ^k _t +  \frac{1}{\delta _i} \sum_{j=1}^n  g^{ij} (\textbf{y} _t) \eta^j _t \right] dt \; . \end{array}
\right.
\end{equation}
Supplemented by the equations defining the noise processes $\eta^j$ (see equation~\eqref{eq:harmonicnoise}), these equations become the SDE system we study in this article.

\section{Derivation of Limiting Equation}

We study the limit of the system~\eqref{MainSDE} as the time delays $\delta_i$ and the correlation times of the colored noises go to zero.  We take each $\eta^j$ to be a stationary harmonic noise process \cite{Schimansky-Geier} defined as the stationary solution of the SDE
\begin{equation}\label{eq:harmonicnoise}
\left\{
\begin{array}{ccl}
\displaystyle d \eta_t^j & = & \displaystyle \frac{1}{\tau _j} \frac{\Gamma}{\Omega ^2} z ^j _t dt \\[12pt]
\displaystyle dz ^j _t & = & \displaystyle - \frac{1}{\tau _j} \frac{\Gamma ^2}{\Omega ^2} z ^j _t dt - \frac{1}{\tau _j} \Gamma \eta ^j _t dt + \frac{1}{\tau _j}\Gamma dW ^j _t
\end{array}
\right.
\end{equation}
where $\Gamma>0$ and $\Omega$ are constants, $\bm{W}_t = (W^1 _t, ..., W^j _t, ..., W^n _t)^{\rm T}$ is an $n$-dimensional Wiener process, and $\tau _j$ is the correlation time of the Ornstein-Uhlenbeck process obtained by taking the limit $\Gamma,\, \Omega ^2 \rightarrow \infty$ while keeping $\frac{\Gamma}{\Omega ^2}$ constant.  The system~\eqref{eq:harmonicnoise} has a unique stationary measure.  The distribution of the system's solution with an arbitrary (nonrandom) initial condition  converges to this stationary measure as $t \to \infty$.  The solution with the initial condition distributed according to the stationary measure defines a stationary process, whose realizations will play the role of colored noise in the SDE system~\eqref{MainSDE}.   We note that as $\tau _j \rightarrow 0$, the component  $\eta^j$  of the solution of equation~\eqref{eq:harmonicnoise} converges to a white noise (see the Appendix for details).

In taking the limit as the delay times $\delta_i$ and the noise correlation times $\tau_j$ go to zero, we assume that all the $\delta_i$ and $\tau_j$ stay proportional to a single characteristic time $\epsilon > 0$.  That is, we let $\delta _i = c_i \epsilon$ and $\tau _j = k _j \epsilon$ where $c_i, k_j > 0$ remain constant in the limit $\delta _i, \tau _j , \epsilon \rightarrow 0$.

We consider the solution to equations~\eqref{MainSDE} and \eqref{eq:harmonicnoise} on a bounded time interval $0 \leq t \leq T$.  We let $(\Omega, \mathcal{F}, P)$ denote the underlying probability space.  We will use the filtration $\{\mathcal{F}_t : t \geq 0 \}$ on $(\Omega,\mathcal{F},P)$ where $\mathcal{F}_t$ is (the usual augmentation of) $\sigma(\{\bm{W}_s:s\leq t\})$, i.e. the $\sigma$-algebra generated by the Wiener process $\bm{W}$ up to time $t$.    

Throughout this article, for an arbitrary vector $\bm{a} \in \mathbb{R}^d$, $\| \bm{a} \|$ will denote its Euclidean norm, and for a matrix $\bm{A} \in \mathbb{R}^{d \times d}$, $\| \bm{A} \|$ will denote the matrix norm induced by the Euclidean norm on $\mathbb{R}^d$.

\begin{theorem}  Suppose that the $f^i$ are bounded functions with bounded, continuous first derivatives and bounded second derivatives and that the $g^{ij}$ are bounded functions with bounded, continuous first derivatives.  Let 
$(\bm{y}^{\epsilon}, \bm{v}^{\epsilon}, \bm{\eta}^{\epsilon}, \bm{z}^{\epsilon})$ solve equations~\eqref{MainSDE} and \eqref{eq:harmonicnoise} (which depend on $\epsilon$ through $\delta_i, \tau_j$) on $0 \leq t \leq T$ with initial conditions $(\bm{y}_0, \bm{v}_0, \bm{\eta} ^{\epsilon} _0, \bm{z} ^{\epsilon} _0)$, where $(\bm{y}_0, \bm{v}_0)$ is the same for every $\epsilon$ and $(\bm{\eta} ^{\epsilon} _0, \bm{z} ^{\epsilon} _0)$ is distributed according to the stationary distribution corresponding to equation~\eqref{eq:harmonicnoise}.  Let $\bm{y}$ solve     
\begin{align} \label{thm limiting equation}
d y^i _t = f^i (\bm{y} _t) dt + \sum _{p,j} g^{pj} (\bm{y} _t) \frac{\partial g^{ij} (\bm{y} _t)}{\partial y_p} & \left[ \frac{\frac{\Gamma}{\Omega ^2} \frac{\delta _p}{\tau _j} + \frac{1}{\Gamma} \left(1 - \frac{\delta _p}{\tau_j} \right)}{2 \left( \frac{\Gamma}{\Omega ^2} \frac{\delta _p}{\tau _j} \left(1 + \frac{\delta _p}{\tau _j} \right) + \frac{1}{\Gamma} \right) } \right] dt \\[1em]
&+ \sum_j g^{ij} (\bm{y}_t) dW^j _t \notag
\end{align}  
on $0 \leq t \leq T$ with the same initial condition~$\bm{y}_0$, and suppose strong uniqueness holds on $0 \leq t \leq T$ for (\ref{thm limiting equation}) with the initial condition $\bm{y}_0$ (strong uniqueness is implied, for example, by the additional assumption that the $g^{ij}$ have bounded second derivatives).  Then
\begin{equation}
\lim _{\epsilon \rightarrow 0} P \left[ \sup_{0 \leq t \leq T} \| \bm{y}^{\epsilon}_t - \bm{y} _t \| > a \right] = 0
\end{equation}
for every $a>0$.

\end{theorem}

\vspace{5pt}

\begin{remark}
Taking the limit $\Gamma, \Omega ^2 \rightarrow \infty$ in equation~\eqref{thm limiting equation} while keeping $\frac{\Gamma}{\Omega ^2}$ constant, we get the simpler limiting equation
\begin{equation} \label{thm limiting equation 2}
d y^i _t = f^i (\bm{y} _t) dt + \sum _{p,j} g^{pj} (\bm{y} _t) \frac{\partial g^{ij} (\bm{y} _t)}{\partial y_p} {1\over2} \left( 1 + \frac{\delta _p}{\tau _j} \right)^{-1} dt + \sum_j g^{ij} (\bm{y}_t) dW^j_t \; .
\end{equation}
\end{remark}

\vspace{5pt}

\begin{remark} Our choice of the distribution of the initial condition $(\bm{\eta} ^{\epsilon} _0, \bm{z} ^{\epsilon} _0)$ is the only one that makes the noise process $(\bm{\eta} ^{\epsilon} , \bm{z} ^{\epsilon} )$ stationary---physically a very natural assumption.  However, the proof of Theorem 1 applies to any choice of $(\bm{\eta} ^{\epsilon} _0, \bm{z} ^{\epsilon} _0)$ such that $E[ \| \bm{\eta} ^{\epsilon} _0 \| ^2]$ and $E[ \| \bm{z}^{\epsilon} _0 \| ^2]$ do not grow faster than $1 / \epsilon$ as $\epsilon \rightarrow 0$.
\end{remark}

\bigskip

\noindent\emph{Outline of the proof of Theorem 1.}  The proof uses the method of Hottovy {\emph et al.} \cite{Hottovy}.  The main tool that we use is a theorem by Kurtz and Protter about convergence of stochastic integrals.  In Section 3.1 we write equations~\eqref{MainSDE} and \eqref{eq:harmonicnoise} together in the matrix form that is used in the Kurtz-Protter theorem.  The theorem itself is stated in Section 3.2.  In Section 3.3 we use it to derive the limiting equations (\ref{thm limiting equation}) and (\ref{thm limiting equation 2}).  The key steps are integrating by parts and then rewriting a certain differential by solving a Lyapunov matrix equation.  In Section 4 we verify that the assumptions of the Kurtz-Protter theorem are satisfied, thus completing the proof of Theorem 1.

\subsection{Matrix form} We introduce the vector process
$$
\bm{X}^{\epsilon} = (\bm{y}^{\epsilon}, \bm{\xi}^{\epsilon}, \bm{\zeta}^{\epsilon}) \; ,
$$
where, as in the statement of the theorem,  $(\bm{y}^{\epsilon}, \bm{v}^{\epsilon}, \bm{\eta}^{\epsilon}, \bm{z}^{\epsilon})$ solves equations~\eqref{MainSDE} and \eqref{eq:harmonicnoise}, $\bm{\xi}^{\epsilon}_t = ((\xi^{\epsilon}_t)_1, \; \dots \; , (\xi^{\epsilon}_t)_n)$ where $(\xi^{\epsilon}_t)_j = \int_0^t (\eta^{\epsilon}_s)_j \,ds$, and
$\bm{\zeta}^{\epsilon}_t = ((\zeta^{\epsilon}_t)_1, \; \dots \; , (\zeta^{\epsilon}_t)_n)$ where $(\zeta^{\epsilon}_t)_j = \int_0^t (z^{\epsilon}_s)_j \, ds = \tau _j \frac{\Omega ^2}{\Gamma} \left[ (\eta^{\epsilon}_t)_j - (\eta^{\epsilon}_0)_j \right]$.  We let $\bm{V}^{\epsilon}_t = \dot{\bm{X}}^{\epsilon}_t$, so that $\bm{V}^{\epsilon}_t = (\bm{v}^{\epsilon}_t, \bm{\eta}^{\epsilon}_t, \bm{z}^{\epsilon}_t)$.  
Equations \eqref{MainSDE} and \eqref{eq:harmonicnoise} can be written in terms of the processes $\bm{X}^{\epsilon}$ and $\bm{V}^{\epsilon}$ as
\begin{equation}\label{vector form}
\left\{
\begin{array}{ccl}
\displaystyle d \bm{X}^{\epsilon} _t &=&  \displaystyle \bm{V}^{\epsilon} _t dt \\[12pt]
\displaystyle d \bm{V}^{\epsilon}  _t & = & \displaystyle \left[ \frac{\bm{F}(\bm{X}^{\epsilon} _t)}{\epsilon} - \frac{\bm{\gamma} (\bm{X}^{\epsilon} _t)}{\epsilon} \bm{V}^{\epsilon} _t + \bm{\kappa} (\bm{X}^{\epsilon} _t) \bm{V} ^{\epsilon} _t \right] dt + \frac{\bm{\sigma}}{\epsilon} d \bm{W} _t
\end{array}
\right.
\end{equation}
where $\bm{F}(\bm{X}^{\epsilon} _t)$ is the vector of length $m + 2n$ that is given, in block form, by
$$ \bm{F}(\bm{X}^{\epsilon} _t) =
\left[\begin{array}{c}
\hat{\bm{f}}( \bm{y^{\epsilon} _t}) \\
\bm{0} \\
\bm{0}
\end{array}\right]
$$
where $\hat{\bm{f}}( \bm{y}^{\epsilon} _t) = \left(\frac{f^1 ( \bm{y}^{\epsilon} _t)}{c_1}, \dots, \frac{f ^m ( \bm{y}^{\epsilon} _t)}{c_m} \right)^T$; $\bm{\gamma}(\bm{X}^{\epsilon} _t)$ is the $(m + 2n) \; \times \; (m + 2n)$ matrix that is given, in block form, by
\begin{equation}\label{eq:10}
 \bm{\gamma}(\bm{X}^{\epsilon} _t) =
\left[
\setlength{\extrarowheight}{5pt} 
\begin{array}{ccc}
\bm{D}^1 & - \hat{ \bm{g}}(\bm{y}^{\epsilon} _t) & \bm{0}  \\
\bm{0} & \bm{0} & - \frac{\Gamma}{\Omega ^2} \bm{D} ^2 \\
\bm{0} &  \Gamma \bm{D} ^2 &  \frac{\Gamma ^2}{\Omega ^2} \bm{D}^2
\end{array} 
\setlength{\extrarowheight}{5pt} 
\right]
\end{equation}
where 
\begin{equation*}
(\hat{ \bm{g}}(\bm{y}^{\epsilon} _t))_{ij} = \frac{g^{ij} (\bm{y}^{\epsilon} _t)}{c_i} \; ,
\end{equation*}
\begin{equation*} \bm{D}^1 =
\left[\begin{array}{cccc}
\frac{1}{c_1} & 0 & ... & 0  \\
0 & \frac{1}{c_2} & ... & 0\\
\vdots & \vdots & \ddots & \vdots \\
0 & 0 & ... & \frac{1}{c_m}
\end{array}\right],
\end{equation*}
and
\begin{equation*} \bm{D}^2 =
\left[\begin{array}{cccc}
\frac{1}{k_1} & 0 & ... & 0  \\
0 & \frac{1}{k_2} & ... & 0\\
\vdots & \vdots & \ddots & \vdots \\
0 & 0 & ... & \frac{1}{k_n}
\end{array}\right];
\end{equation*}
$\bm{\kappa}(\bm{X}^{\epsilon} _t)$ is the $(m + 2n) \; \times \; (m + 2n)$ matrix that is given, in block form, by
\begin{equation*} \bm{\kappa}(\bm{X}^{\epsilon} _t) =
\left[\begin{array}{ccc}
\hat{\bm{J}}_f (\bm{y}^{\epsilon} _t) & \bm{0} & \bm{0}  \\
\bm{0} & \bm{0} & \bm{0}\\
\bm{0} & \bm{0} & \bm{0}
\end{array}\right]
\end{equation*}
where 
\begin{equation*} \hat{\bm{J}}_f (\bm{y}^{\epsilon} _t) =
\left[
\setlength{\extrarowheight}{7pt} 
\begin{array}{cccc}
\frac{c_1}{c_1} \frac{\partial f^1 (\bm{y}^{\epsilon} _t)}{\partial y _1} & \frac{c_2}{c_1} \frac{\partial f^1 (\bm{y}^{\epsilon} _t)}{\partial y _2} & ... & \frac{c_m}{c_1} \frac{\partial f^1 (\bm{y}^{\epsilon} _t)}{\partial y _m}  \\
\frac{c_1}{c_2} \frac{\partial f^2 (\bm{y}^{\epsilon} _t)}{\partial y _1} & \frac{c_2}{c_2} \frac{\partial f^2 (\bm{y}^{\epsilon} _t)}{\partial y _2} & ... & \frac{c_m}{c_2} \frac{\partial f^2 (\bm{y}^{\epsilon} _t)}{\partial y _m}\\
\vdots & \vdots & \ddots & \vdots \\
\frac{c_1}{c_m} \frac{\partial f^m (\bm{y}^{\epsilon} _t)}{\partial y _1} & \frac{c_2}{c_m} \frac{\partial f^m (\bm{y}^{\epsilon} _t)}{\partial y _2} & ... & \frac{c_m}{c_m} \frac{\partial f^m (\bm{y}^{\epsilon} _t)}{\partial y _m}
\end{array}
\setlength{\extrarowheight}{7pt} 
\right];
\end{equation*}
$\bm{\sigma}$ is the $(m + 2n) \; \times \; n$ matrix that is given, in block form, by
$$ \bm{\sigma} =
\left[\begin{array}{c}
\bm{0}  \\
\bm{0} \\
\Gamma \bm{D}^2
\end{array}\right]
$$
and $\bm{W}$ is the $n$-dimensional Wiener process in equation~\eqref{eq:harmonicnoise}. Using the introduced notation, we obtain the desired matrix form of equations~\eqref{MainSDE} and \eqref{eq:harmonicnoise}. The equation for $\bm{V}^{\epsilon}_t$ becomes
\begin{equation*}
\left[\bm{\gamma}(\bm{X}^{\epsilon} _t) - \epsilon \bm{\kappa} (\bm{X}^{\epsilon} _t) \right] \bm{V} ^{\epsilon} _t dt =  \bm{F}(\bm{X}^{\epsilon} _t) dt + \bm{\sigma} d \bm{W} _t - \epsilon d \bm{V}^{\epsilon} _t .
\end{equation*}
By Lemma 2 in Section 4, for $\epsilon$ sufficiently small, $\bm{\gamma}(\bm{X}^{\epsilon} _t) - \epsilon \bm{\kappa} (\bm{X}^{\epsilon} _t)$ is invertible.  Thus, for $\epsilon$ sufficiently small, we can solve for $\bm{V}^{\epsilon} _t dt$, rewriting the equation for $\bm{X}^{\epsilon} _t$ as
\begin{equation*}
d \bm{X}^{\epsilon} _t =  \bm{V}^{\epsilon} _t dt = \left( \bm{\gamma}(\bm{X}^{\epsilon} _t) - \epsilon \bm{\kappa} (\bm{X}^{\epsilon} _t) \right) ^{-1} \left[ \bm{F}(\bm{X}^{\epsilon} _t ) dt + \bm{\sigma} d \bm{W} _t - \epsilon d \bm{V}^{\epsilon} _t \right] .
\end{equation*}
In integral form, this equation is
\begin{align} \label{integral form}
 \bm{X}^{\epsilon} _t = \bm{X}_0 &+ \int_0^t (\bm{\gamma}(\bm{X}^{\epsilon} _s) - \epsilon \bm{\kappa} (\bm{X}^{\epsilon} _s))^{-1} \bm{F}(\bm{X}^{\epsilon} _s) ds \notag \\
 &+ \int_0^t (\bm{\gamma}(\bm{X}^{\epsilon} _s) - \epsilon \bm{\kappa} (\bm{X}^{\epsilon} _s))^{-1} \bm{\sigma} d \bm{W} _s \\
 &- \int_0^t  (\bm{\gamma}(\bm{X}^{\epsilon} _s) - \epsilon \bm{\kappa} (\bm{X}^{\epsilon} _s))^{-1} \epsilon d \bm{V}^{\epsilon} _s \notag 
\end{align}
where $\bm{X}_0 = (\bm{y}_0, \bm{0}, \bm{0})$ is independent of $\epsilon$ due to our assumption that $\bm{y}_0$ is the same for all $\epsilon$.

\begin{remark}
The equations in~\eqref{vector form} have a structure similar to equations studied in Ref.~\cite{Hottovy}, except for the additional term $\bm{\kappa} (\bm{X}^{\epsilon}_t) \bm{V} ^{\epsilon} _t\,dt$.  The method of Ref.~\cite{Hottovy} will be suitably adapted to treat this term and to account for the structure of the other terms in the second equation in~\eqref{vector form}.
\end{remark}

\subsection{Convergence of stochastic integrals}
We use a theorem of Kurtz and Protter \cite{Kurtz} which, for greater clarity, we state here in a less general but sufficient form. 
Let $\{\mathcal{F}_t : t \geq 0 \}$ be a filtration on a probability space $(\Omega, \mathcal{F}, P)$.  In our case  $\mathcal{F}_t$ will be the usual augmentation of  $\sigma(\{\bm{W}_s:s\leq t\})$ (the $\sigma$-algebra generated by the Wiener process $\bm{W}$ up to time $t$) introduced earlier.  The processes we consider below are assumed to be adapted to this filtration.
We consider a family of pairs of processes $(\bm{U}^{\epsilon}, \bm{H}^{\epsilon})$ where $\bm{U}^{\epsilon}$ has paths in $C([0,T], \mathbb{R}^{m + 2n})$ (i.e. the space of continuous functions from $[0,T]$ to $\mathbb{R}^{m + 2n}$) and where $\bm{H}^{\epsilon} $ is a semimartingale with paths in $C([0,T], \mathbb{R}^d)$.  Let $\bm{H}^{\epsilon}  = \bm{M}^{\epsilon}  + \bm{A}^{\epsilon} $ be the Doob-Meyer decomposition of $\bm{H}^{\epsilon} $ so that $\bm{M}^{\epsilon} $ is a local martingale and $\bm{A}^{\epsilon} $ is a process of locally bounded variation \cite{Revuz}.  We denote the total variation of $\bm{A}^{\epsilon}$ by $V(\bm{A}^{\epsilon})$.  Let  $\bm{h}$ and $\bm{h}^{\epsilon}:\mathbb{R}^{m + 2n} \rightarrow \mathbb{R}^{(m + 2n) \times d}$, $\epsilon > 0$, be a family of matrix-valued functions.  Suppose that the process $\bm{Y}^{\epsilon}$, with paths in $C([0,T], \mathbb{R}^{m + 2n})$, satisfies the stochastic integral equation
\begin{equation}
\bm{Y}^{\epsilon}_t = \bm{Y}_0 + \bm{U}^{\epsilon} _t + \int _0 ^t \bm{h}^{\epsilon}(\bm{Y}^{\epsilon} _s) d \bm{H}^{\epsilon} _s
\end{equation}
with $\bm{Y} _0$ independent of $\epsilon$.  Let $\bm{H} $ be a semimartingale with paths in $C([0,T], \mathbb{R}^d)$  and let $\bm{Y}$, with paths in $C([0,T], \mathbb{R}^{m + 2n})$, satisfy the stochastic integral equation
\begin{equation}
\label{KP equation 2}
\bm{Y}_t = \bm{Y} _0 + \int_0^t \bm{h} (\bm{Y}_s) d\bm{H} _s \; .
\end{equation}

\begin{lemma}[\textbf{\cite[Theorem 5.4 and Corollary 5.6]{Kurtz}}]
Suppose $(\bm{U}^{\epsilon}, \bm{H}^{\epsilon}) \rightarrow (\bm{0}, \bm{H})$ in probability with respect to $C([0,T], \mathbb{R}^{m + 2n} \times \mathbb{R}^d)$, i.e. for all $a > 0$,
\begin{equation} \label{KP assumption}
P \left[ \sup_{0 \leq s \leq T} \big(\| \bm{U}^{\epsilon}_s \| + \| \bm{H}^{\epsilon} _s - \bm{H}_s \|\big) > a \right] \rightarrow 0
\end{equation}
as $\epsilon \rightarrow 0$, and the following conditions are satisfied:

\begin{condition}
For every $t \in [0,T]$, the family of total variations evaluated at $t$,  $\{V_t(\bm{A}^{\epsilon}) \}$, is stochastically bounded, i.e. $P[ V_t (\bm{A}^{\epsilon}) > L] \rightarrow 0$ as $L \rightarrow \infty$, uniformly in $\epsilon$.
\end{condition}

\begin{condition}

\begin{enumerate}
\item $\sup_{\theta \in \mathbb{R}^{m + 2n}} \| \bm{h}^{\epsilon} (\theta) - \bm{h} (\theta) \| \rightarrow 0$ as $\epsilon \rightarrow 0$
\item $\bm{h}$ is continuous (see \cite[Example 5.3]{Kurtz})
\end{enumerate}

\end{condition}

Suppose that there exists a strongly unique global solution to equation~\eqref{KP equation 2}.  Then, as $\epsilon\rightarrow~0$,  $\bm{Y}^{\epsilon}~\rightarrow~\bm{Y}$ in probability with respect to  $C([0,T], \mathbb{R}^{m + 2n})$, i.e. for all $a > 0$,
$$P \left[ \sup_{0 \leq s \leq T} \| \bm{Y}^{\epsilon}_s - \bm{Y} _s \| > a \right] \rightarrow 0 \; \; \mathrm{as} \; \; \epsilon \rightarrow 0 \; .$$
\end{lemma}

\subsection{Proof of Theorem 1}
We cannot apply Lemma~1 directly to equation~\eqref{integral form} because $\epsilon \bm{V}^{\epsilon}$ does not satisfy Condition~1. Instead, we integrate by parts the $i^{\mathrm{th}}$ component of the last integral in equation~\eqref{integral form}:
\begin{align} \label{rightafterbyparts}
 \int_ 0^t \sum_j \big(& (\bm{\gamma}(\bm{X}^{\epsilon} _s) - \epsilon \bm{\kappa} (\bm{X}^{\epsilon} _s))^{-1} \big) _{ij} \epsilon d(\bm{V}^{\epsilon} _s)_j =  \hspace{200pt} \notag \\
 & \sum_j \big( (\bm{\gamma}(\bm{X}^{\epsilon} _t) - \epsilon \bm{\kappa} (\bm{X}^{\epsilon} _t))^{-1} \big) _{ij} \epsilon (\bm{V}^{\epsilon} _t)_j 
- \sum_j \big( (\bm{\gamma}(\bm{X}_0) - \epsilon \bm{\kappa} (\bm{X}_0))^{-1} \big) _{ij} \epsilon (\bm{V} ^{\epsilon} _0)_j \notag \\
&- \int_0^t \sum_{\ell , j} \frac{\partial}{\partial X_{\ell}} \big[ \big( (\bm{\gamma}(\bm{X}^{\epsilon} _s) - \epsilon \bm{\kappa} (\bm{X}^{\epsilon} _s))^{-1} \big) _{ij} \big] \epsilon (\bm{V}^{\epsilon} _s)_j d(\bm{X}^{\epsilon} _s)_{\ell}
\end{align}
where $\bm{V} ^{\epsilon} _0 = (\bm{v}_0, \bm{\eta} ^{\epsilon} _0, \bm{z} ^{\epsilon} _0)$.  Note that
$$
d \big[ \big( (\bm{\gamma}(\bm{X}^{\epsilon} _s) - \epsilon \bm{\kappa} (\bm{X}^{\epsilon} _s))^{-1} \big) _{ij} \big] = \sum_{\ell} \frac{\partial}{\partial X_{\ell}} \big[ \big( (\bm{\gamma}(\bm{X}^{\epsilon} _s) - \epsilon \bm{\kappa} (\bm{X}^{\epsilon} _s))^{-1} \big) _{ij} \big]   d(\bm{X}^{\epsilon} _s)_{\ell}
$$
because $\bm{X}^{\epsilon} _s$ is continuously differentiable. The It\^o term in the integration by parts formula is zero for a similar reason.

Since $d(\bm{X}^{\epsilon} _s)_{\ell} = (\bm{V}^{\epsilon} _s)_{\ell} \; ds$, we can write the last integral in equation~\eqref{rightafterbyparts} as
$$ 
\int_0^t \sum_{\ell , j} \frac{\partial}{\partial X_{\ell}} \big[ \big( (\bm{\gamma}(\bm{X}^{\epsilon} _s) - \epsilon \bm{\kappa} (\bm{X}^{\epsilon} _s))^{-1} \big) _{ij} \big] \epsilon (\bm{V}^{\epsilon} _s)_j (\bm{V}^{\epsilon} _s)_{\ell} \; ds \; .
$$
The product $\epsilon (\bm{V}^{\epsilon} _s)_j (\bm{V}^{\epsilon} _s)_{\ell}$ that appears in the above integral is the $(j, \ell)$ entry of the outer product matrix $\epsilon \bm{V}^{\epsilon} _s (\bm{V}^{\epsilon} _s)^{\rm T}$.  Our next step is to express this matrix as the solution of a certain equation.
We start by using the It\^o product formula to calculate
$$
d[\epsilon \bm{V}^{\epsilon} _s (\epsilon \bm{V}^{\epsilon} _s)^{\rm T}] = \epsilon (d( \bm{V}^{\epsilon} _s))(\epsilon \bm{V}^{\epsilon} _s)^{\rm T} + \epsilon \bm{V}^{\epsilon} _s (\epsilon d(\bm{V}^{\epsilon} _s)^{\rm T}) + d(\epsilon \bm{V}^{\epsilon} _s)d(\epsilon \bm{V}^{\epsilon} _s)^{\rm T},
$$
so that, using equation~\eqref{vector form},
\begin{align}
\label{differential}
d[\epsilon \bm{V}^{\epsilon} _s (\epsilon \bm{V}^{\epsilon} _s)^{\rm T}] & =  [\epsilon \bm{F} (\bm{X}^{\epsilon} _s) (\bm{V}^{\epsilon} _s)^{\rm T} - \epsilon \bm{\gamma} (\bm{X}^{\epsilon} _s)\bm{V}^{\epsilon} _s (\bm{V}^{\epsilon} _s)^{\rm T} + \epsilon ^2 \bm{\kappa} (\bm{X}^{\epsilon} _s)\bm{V}^{\epsilon} _s (\bm{V}^{\epsilon} _s)^{\rm T}] ds \notag \\
 & +  \epsilon \bm{\sigma} d \bm{W}_s (\bm{V}^{\epsilon} _s)^{\rm T}  \\
 & +  [\epsilon \bm{V}^{\epsilon} _s (\bm{F} (\bm{X}^{\epsilon} _s))^{\rm T}  - \epsilon \bm{V}^{\epsilon} _s (\bm{V}^{\epsilon} _s)^{\rm T} (\bm{\gamma} (\bm{X}^{\epsilon} _s))^{\rm T} + \epsilon ^2 \bm{V}^{\epsilon} _s (\bm{V}^{\epsilon} _s)^{\rm T} (\bm{\kappa} (\bm{X}^{\epsilon} _s))^{\rm T}]ds \notag \\
 & +  \epsilon \bm{V}^{\epsilon} _s (\bm{\sigma} d \bm{W}_s)^{\rm T} + \bm{\sigma} \bm{\sigma}^{\rm T} ds \; . \notag
\end{align}
Defining
\begin{equation} \label{Utildeequation}
\tilde{\bm{U}}^{\epsilon} _t = \int _0^t [ \epsilon \bm{V}^{\epsilon} _s (\bm{F} (\bm{X}^{\epsilon} _s))^{\rm T} + \epsilon \bm{V}^{\epsilon} _s (\epsilon \bm{V}^{\epsilon} _s)^{\rm T} (\bm{\kappa} (\bm{X}^{\epsilon} _s))^{\rm T}] ds + \int_0 ^t \epsilon \bm{V}^{\epsilon} _s (\bm{\sigma} d \bm{W}_s)^{\rm T}
\end{equation}
and combining (\ref{differential}) and (\ref{Utildeequation}), we obtain
\begin{align} \label{rightafterbyparts2}
& - \epsilon \bm{V}^{\epsilon} _t (\bm{V}^{\epsilon} _t)^{\rm T} (\bm{\gamma} (\bm{X}^{\epsilon} _t))^{\rm T} dt - \epsilon \bm{\gamma} (\bm{X}^{\epsilon} _t)\bm{V}^{\epsilon} _t (\bm{V}^{\epsilon} _t)^{\rm T} dt  \notag \\
&= d[\epsilon \bm{V}^{\epsilon} _t (\epsilon \bm{V}^{\epsilon} _t)^{\rm T}] - \bm{\sigma} \bm{\sigma}^{\rm T} dt - d \tilde{\bm{U}} ^{\epsilon} _t - d (\tilde{\bm{U}} ^{\epsilon} _ t )^{\rm T} \; .
\end{align}
Our goal is to write the differential $ \epsilon \bm{V}^{\epsilon} _t (\bm{V}^{\epsilon} _t)^{\rm T} dt$ in another form and substitute it back into equation~\eqref{rightafterbyparts}.  Letting $\Delta t > 0$, we integrate (\ref{rightafterbyparts2}) to obtain
\begin{align} 
\label{integraleqafterbyparts1}
& - \int _t ^{t + \Delta t} \epsilon \bm{V}^{\epsilon} _s (\bm{V}^{\epsilon} _s)^{\rm T} (\bm{\gamma} (\bm{X}^{\epsilon} _s))^{\rm T} ds - \int_t ^{t + \Delta t} \epsilon \bm{\gamma} (\bm{X}^{\epsilon} _s)\bm{V}^{\epsilon} _s (\bm{V}^{\epsilon} _s)^{\rm T} ds  \notag \\
&= \int _t ^{t + \Delta t} \left( d[\epsilon \bm{V}^{\epsilon} _s (\epsilon \bm{V}^{\epsilon} _s)^{\rm T}] - \bm{\sigma} \bm{\sigma}^{\rm T} ds - d \tilde{\bm{U}} ^{\epsilon} _s - d (\tilde{\bm{U}} ^{\epsilon} _ s)^{\rm T} \right) \; .
\end{align}
Defining
$$\bm{E} (t, \Delta t) = \int _t ^{t + \Delta t} \epsilon \bm{V}^{\epsilon} _s (\bm{V}^{\epsilon} _s)^{\rm T} (\bm{\gamma} (\bm{X}^{\epsilon} _s))^{\rm T} ds - \epsilon \bm{V}^{\epsilon} _t (\bm{V}^{\epsilon} _t)^{\rm T} (\bm{\gamma} (\bm{X}^{\epsilon} _t))^{\rm T} \Delta t \; ,$$
we write (\ref{integraleqafterbyparts1}) as
\begin{align}
\label{integraleqafterbyparts2}
& - \epsilon \bm{V}^{\epsilon} _t (\bm{V}^{\epsilon} _t)^{\rm T} (\bm{\gamma} (\bm{X}^{\epsilon} _t))^{\rm T} \Delta t - \epsilon \bm{\gamma} (\bm{X}^{\epsilon} _t)\bm{V}^{\epsilon} _t (\bm{V}^{\epsilon} _t)^{\rm T} \Delta t  \notag \\
&= \int _t ^{t + \Delta t} \left( d[\epsilon \bm{V}^{\epsilon} _s (\epsilon \bm{V}^{\epsilon} _s)^{\rm T}] - \bm{\sigma} \bm{\sigma}^{\rm T} ds - d \tilde{\bm{U}} ^{\epsilon} _s - d (\tilde{\bm{U}} ^{\epsilon} _ s)^{\rm T} \right) + \bm{E} (t, \Delta t) + (\bm{E} (t, \Delta t))^T \; .
\end{align}  
Letting $\bm{A} = - \bm{\gamma} (\bm{X}^{\epsilon} _t)$, $\bm{B} = \epsilon \bm{V}^{\epsilon} _t (\bm{V}^{\epsilon} _t)^{\rm T} \Delta t$, and 
$$\bm{C} = \int _t ^{t + \Delta t} \left( d[\epsilon \bm{V}^{\epsilon} _s (\epsilon \bm{V}^{\epsilon} _s)^{\rm T}] - \bm{\sigma} \bm{\sigma}^{\rm T} ds - d \tilde{\bm{U}} ^{\epsilon} _s - d (\tilde{\bm{U}} ^{\epsilon} _ s)^{\rm T} \right) \; + \; \bm{E} (t, \Delta t) \; + \; (\bm{E} (t, \Delta t)) ^T \; ,$$ 
equation~\eqref{integraleqafterbyparts2} becomes
\begin{equation*}
\bm{A}\bm{B} + \bm{B}\bm{A}^{\rm T} = \bm{C} \; .
\end{equation*}
An equation of this form  (to be solved for $\bm{B}$) is called Lyapunov's equation \cite{Bellman, Ortega}. By Ref.~\cite[Theorem~6.4.2]{Ortega}, if the real parts of all eigenvalues of $\bm{A}$ are negative, it has a unique solution 
\begin{equation*}
\bm{B} = - \int_0 ^{\infty} e^{\bm{A}y}\bm{C} e^{\bm{A}^{\rm T} y} dy
\end{equation*}
for any $\bm{C}$.  The eigenvalues of $\bm{\gamma} (\bm{X}^{\epsilon} _t)$ are 
\begin{equation}
\label{eigenvaluesgamma}
 \frac{1}{c_i}, \hspace{5pt} i = 1, ..., m, \hspace{15pt} \mathrm{and} \hspace{15pt} \frac{\Gamma ^2}{2k_j \Omega ^2} \left[1 \pm \sqrt{1 - 4 \frac{\Omega ^2}{\Gamma ^2}} \right], \hspace{5pt}  j = 1, ..., n\; ;
 \end{equation}
 in particular, they do not depend on $\bm{X}^{\epsilon} _t$ and have positive real parts (since $c_i > 0$ and $k_j > 0$ for $i = 1,..., m, \; j = 1, ..., n$).  Thus, all eigenvalues of $\bm{A} = - \bm{\gamma} (\bm{X}^{\epsilon} _t)$ have negative real parts, so we have
\begin{align*}
 \epsilon \bm{V}^{\epsilon} _t (\bm{V}^{\epsilon} _t)^{\rm T} \Delta t  = -& \int_0^{\infty}  e^{-\bm{\gamma} (\bm{X}^{\epsilon} _t)y} \Bigg( \int _t ^{t + \Delta t} \Big( d[\epsilon \bm{V}^{\epsilon} _s (\epsilon \bm{V}^{\epsilon} _s)^{\rm T}] - \bm{\sigma} \bm{\sigma}^{\rm T} ds \\
& - d \tilde{\bm{U}} ^{\epsilon} _s - d (\tilde{\bm{U}} ^{\epsilon} _ s)^{\rm T} \Big) + \bm{E} (t, \Delta t) + ( \bm{E} (t, \Delta t))^T \Bigg) e^{-(\bm{\gamma} (\bm{X}^{\epsilon} _t))^{\rm T} y} dy \; .
\end{align*}
Now, for $0 \leq a < b \leq T$ and $N \in \mathbb{N}$, let $\Delta t = (b - a) / N$ and let $\{t_i : 0 \leq i \leq N \}$ be the partition of $[a, b ]$ such that $t_0 = a$, $t_N = b$, and $t_{i + 1} - t_i = \Delta t$ for $0 \leq i \leq N - 1$.  Then
\begin{align*}
& \sum _{i = 0} ^{N - 1} \epsilon  \left( \bm{V}^{\epsilon} _{t _i} (\bm{V}^{\epsilon} _{t_i})^{\rm T} \right) _{j \ell} \Delta t \; = \\
& - \sum _{k_1 , k_2} \int_0^{\infty} ( e^{-\bm{\gamma} (\bm{X}^{\epsilon} _t)y} ) _{j k_1} \Bigg( \int _{a} ^{b} \Big( d[\epsilon \bm{V}^{\epsilon} _s (\epsilon \bm{V}^{\epsilon} _s)^{\rm T} ] - \bm{\sigma} \bm{\sigma}^{\rm T} ds \\
& \hspace{150pt} - d \tilde{\bm{U}} ^{\epsilon} _s  - d (\tilde{\bm{U}} ^{\epsilon} _ s)^{\rm T} \Big) _{k_1 k_2} \Bigg)(e^{-(\bm{\gamma} (\bm{X}^{\epsilon} _t))^{\rm T} y}) _{k_2 \ell} \; dy \\
& - \sum _{k_1 , k_2} \int_0^{\infty} (e^{-\bm{\gamma} (\bm{X}^{\epsilon} _t)y}) _{j k_1} \Bigg( \sum_{i = 0} ^{N - 1} \Big( \bm{E} (t_i, \Delta t) + (\bm{E} (t_i, \Delta t))^T \Big) _{k_1 k_2} \Bigg) (e^{-(\bm{\gamma} (\bm{X}^{\epsilon} _t))^{\rm T} y}) _{k_2 \ell} \; dy \\
& = \; - \sum _{k_1 , k_2} \int _{a} ^{b} \int_0^{\infty} (e^{-\bm{\gamma} (\bm{X}^{\epsilon} _t)y})_{j k_1} (e^{-(\bm{\gamma} (\bm{X}^{\epsilon} _t))^{\rm T} y})_{k_2 \ell} \; dy \Big( d[\epsilon \bm{V}^{\epsilon} _s (\epsilon \bm{V}^{\epsilon} _s)^{\rm T}] \\
& \hspace{200pt} - \bm{\sigma} \bm{\sigma}^{\rm T} ds - d \tilde{\bm{U}} ^{\epsilon} _s - d (\tilde{\bm{U}} ^{\epsilon} _ s)^{\rm T}\Big) _{k_1 k_2} \\
& \hspace{12pt} - \sum _{k_1 , k_2} \int_0^{\infty} (e^{-\bm{\gamma} (\bm{X}^{\epsilon} _t)y})_{j k_1}  \Bigg( \sum_{i = 0} ^{N - 1} \Big( \bm{E} (t_i, \Delta t) + ( \bm{E} (t_i, \Delta t))^T \Big) _{k_1 k_2} \Bigg) (e^{-(\bm{\gamma} (\bm{X}^{\epsilon} _t))^{\rm T} y})_{k_2 \ell} \; dy
\end{align*}
where the second equality follows from the stochastic Fubini's Theorem~\cite[Chapter IV, Theorem 46] {Protter}.
Fix the $\omega$.  Since the corresponding realization of the process $(\bm{X}^{\epsilon} , \bm{V}^{\epsilon})$ is continuous and $\bm{\gamma}$ is continuous, for every  $\alpha > 0$ there exists  $\delta > 0$ (depending on $\omega$) such that 
$$ \| \bm{V}^{\epsilon} _s (\bm{V}^{\epsilon} _s)^{\rm T} (\bm{\gamma} (\bm{X}^{\epsilon} _s))^{\rm T} - \bm{V}^{\epsilon} _u (\bm{V}^{\epsilon} _u)^{\rm T} (\bm{\gamma} (\bm{X}^{\epsilon} _u))^{\rm T} \| \leq \alpha$$
for $|s - u| < \delta$, $s,u \in [a, b]$.  Thus, for $\Delta t < \delta$, we have $ \| \bm{E} (t_i , \Delta t) \| \leq \Delta t \alpha$ for all $0 \leq i \leq N - 1$.  Taking the limit $\Delta t \rightarrow 0$ (i.e. taking $N \rightarrow \infty$), we have
\begin{align*}
& \Bigg\| \int _{a}^{b} \epsilon \left( \bm{V}^{\epsilon} _{s} (\bm{V}^{\epsilon} _{s})^{\rm T} \right) _{j \ell} ds + \sum _{k_1 , k_2} \int _{a} ^{b} \int_0^{\infty} (e^{-\bm{\gamma} (\bm{X}^{\epsilon} _t)y})_{j k_1} \times \\
& \hspace{20pt} (e^{-(\bm{\gamma} (\bm{X}^{\epsilon} _t))^{\rm T} y}) _{k_2 \ell} \; dy \Big( d[\epsilon \bm{V}^{\epsilon} _s (\epsilon \bm{V}^{\epsilon} _s)^{\rm T}] - \bm{\sigma} \bm{\sigma}^{\rm T} ds - d \tilde{\bm{U}} ^{\epsilon} _s - d (\tilde{\bm{U}} ^{\epsilon} _ s)^{\rm T}\Big) _{k_1 k_2}  \Bigg\| \leq C \alpha \; .
\end{align*}
Since $\alpha$ was arbitrary, the expression on the left-hand side is zero for the fixed (arbitrary) $\omega$.  Taking the derivative with respect to $b$, rearranging, and going back to matrix form, we get 
\begin{align*}
\epsilon \bm{V}^{\epsilon} _t (\bm{V}^{\epsilon} _t)^{\rm T} dt  &= - \int_0^{\infty} e^{-\bm{\gamma} (\bm{X}^{\epsilon} _t)y} \Big( d[\epsilon \bm{V}^{\epsilon} _t (\epsilon \bm{V}^{\epsilon} _t)^{\rm T}] \\
 &- \bm{\sigma} \bm{\sigma}^{\rm T} dt - d \tilde{\bm{U}} ^{\epsilon} _t - d (\tilde{\bm{U}} ^{\epsilon} _ t )^{\rm T} \Big) e^{-(\bm{\gamma} (\bm{X}^{\epsilon} _t))^{\rm T} y} dy \\
&= - \underbrace{ \int_0^{\infty} e^{-\bm{\gamma} (\bm{X}^{\epsilon} _t)y} d[\epsilon \bm{V}^{\epsilon} _t (\epsilon \bm{V}^{\epsilon} _t)^{\rm T}] e^{-(\bm{\gamma} (\bm{X}^{\epsilon} _t))^{\rm T} y} dy}_{d \bm{C} ^1 _t} \\
 &+ \underbrace{ \int_0^{\infty} e^{-\bm{\gamma} (\bm{X}^{\epsilon} _t)y} (\bm{\sigma} \bm{\sigma}^{\rm T} dt)  e^{- (\bm{\gamma} (\bm{X}^{\epsilon} _t))^{\rm T} y} dy}_{d \bm{C} ^2 _t} \\
 &+ \underbrace{ \int_0^{\infty} e^{-\bm{\gamma} (\bm{X}^{\epsilon} _t)y} ( d \tilde{\bm{U}} ^{\epsilon} _t + d (\tilde{\bm{U}} ^{\epsilon} _ t )^{\rm T}) e^{-(\bm{\gamma} (\bm{X}^{\epsilon} _t))^{\rm T} y} dy}_{d \bm{C} ^3 _t} .
\end{align*}
After substituting the above expression into equation~\eqref{rightafterbyparts}, a part of the term containing $d \bm{C} ^1 _t$ will be included in the function $\bm{h}^{\epsilon}$ (in the notation of Lemma~1) and the other part will be included in the differential of the $\bm{H}^{\epsilon}$ process. Neither of them will contribute to the limiting equation~\eqref{thm limiting equation}.  The term containing $d \bm{C} ^3 _t$ will become a part of $\bm{U}^{\epsilon}_t$, which will be shown to converge to zero, and so this term will not contribute either.  The  noise-induced drift term will come from the term containing $d \bm{C} ^2 _t$.

First, we have 
\begin{align*}
 (d \bm{C}^1_t) _{j \ell} &= \sum_{k_1 , k_2} \int_0^{\infty} (e^{-\bm{\gamma} (\bm{X}^{\epsilon} _t)y}) _{j k_1} d[(\epsilon \bm{V}^{\epsilon} _t)_{k_1} (\epsilon \bm{V}^{\epsilon} _t)^{\rm T}_{k_2}] (e^{-(\bm{\gamma} (\bm{X}^{\epsilon} _t))^{\rm T} y})_{k_2 \ell} \; dy \\
&= \sum_{k_1 , k_2} d[(\epsilon \bm{V}^{\epsilon} _t)_{k_1} (\epsilon \bm{V}^{\epsilon} _t)^{\rm T}_{k_2}]  \int_0^{\infty} (e^{-\bm{\gamma} (\bm{X}^{\epsilon} _t)y}) _{j k_1} (e^{-(\bm{\gamma} (\bm{X}^{\epsilon} _t))^{\rm T} y})_{k_2 \ell} \; dy \; .
\end{align*}
Next, we have $d \bm{C}^2 _t = \bm{J} (\bm{X}^{\epsilon} _t) dt$ where $\bm{J}$ is the unique solution of the Lyapunov equation
\begin{equation} \label{Lyapunov}
\bm{J} \bm{\gamma} ^{\rm T} + \bm{\gamma} \bm{J} = \bm{\sigma} \bm{\sigma} ^{\rm T} .
\end{equation}
Finally, using equation~\eqref{Utildeequation} for $\tilde{\bm{U}}^{\epsilon}$ we see that 
 \begin{align*}
 (d\bm{C}^3 _t)_{j \ell} &= \sum_{k_1 , k_2} \Bigg[ \int _0^{\infty} (e^{- \bm{\gamma}(\bm{X}^{\epsilon} _t)y})_{j k_1}(e^{- (\bm{\gamma} (\bm{X}^{\epsilon} _t))^{\rm T} y})_{k_2 \ell} \; dy \Big( [\epsilon \bm{V} ^{\epsilon} _t (\bm{F} (\bm{X}^{\epsilon} _t))^{\rm T}] _{k_1 k_2} dt  \notag \\ 
&+ \; [\epsilon  \bm{V}^{\epsilon} _t ( \epsilon \bm{V}^{\epsilon} _t)^{\rm T} (\bm{\kappa} (\bm{X}^{\epsilon} _t))^{\rm T}] _{k_1 k_2} dt \; + \;  [\epsilon \bm{V}^{\epsilon} _t (\bm{\sigma} d \bm{W} _t)^{\rm T} ] _{k_1 k_2}   \notag \\
&+ \; [\bm{F} (\bm{X} ^{\epsilon} _t)(\epsilon \bm{V}^{\epsilon} _t)^{\rm T} ] _{k_1 k_2} dt \; + \; [ \bm{\kappa} (\bm{X}^{\epsilon} _t) \epsilon \bm{V}^{\epsilon} _t (\epsilon \bm{V}^{\epsilon} _t)^{\rm T} ]_{k_1 k_2} dt \notag \\
& + \; [ \bm{\sigma} d \bm{W} _t (\epsilon \bm{V}^{\epsilon} _t)^{\rm T} ] _{k_1 k_2} \Big) \Bigg] \; .
 \end{align*}

We are now ready to rewrite equation~\eqref{integral form} and apply Lemma~1.  After substituting the expression for $\epsilon \bm{V}^{\epsilon} _t (\bm{V} ^{\epsilon} _t)^{\rm T} dt$ into equation~\eqref{rightafterbyparts}, equation~\eqref{integral form} becomes
\begin{align} \label{afterbyparts}
 (\bm{X}^{\epsilon} _t)_i & = (\bm{X}_0)_i + (\bm{U}^{\epsilon} _t)_i + \int_0^t \big( ( \bm{\gamma}(\bm{X}^{\epsilon} _s) - \epsilon \bm{\kappa} (\bm{X}^{\epsilon} _s))^{-1} \bm{F}(\bm{X}^{\epsilon} _s) \big) _i \; ds \notag \\
& + \left( \int_0^t (\bm{\gamma}(\bm{X}^{\epsilon} _s) - \epsilon \bm{\kappa} (\bm{X}^{\epsilon} _s))^{-1} \bm{\sigma}  d \bm{W} _s \right) _i \notag \\
& + \sum_{\ell , j} \int_0^t \frac{\partial}{\partial X_{\ell}} \big[ \big( ( \bm{\gamma}(\bm{X}^{\epsilon} _s) - \epsilon \bm{\kappa} (\bm{X}^{\epsilon} _s))^{-1} \big) _{ij} \big] \bm{J}_{j \ell} (\bm{X} ^{\epsilon} _s) ds \notag \\
& + \sum_{\ell , j} \Bigg[ \int_0^t \frac{\partial}{\partial X_{\ell}} \big[ \big( ( \bm{\gamma}(\bm{X}^{\epsilon} _s) - \epsilon \bm{\kappa} (\bm{X}^{\epsilon} _s))^{-1} \big) _{ij} \big] \; \times  \notag \\
& \sum_{k_1 , k_2} \left(- \int_0^{\infty} (e^{- \bm{\gamma}(\bm{X}^{\epsilon} _s)y})_{jk_1}(e^{- (\bm{\gamma} (\bm{X}^{\epsilon} _s ))^{\rm T} y})_{k_2 \ell} \; dy \right) d[(\epsilon \bm{V}^{\epsilon} _s)_{k_1} (\epsilon \bm{V}^{\epsilon} _s)^{\rm T} _{k_2}] \Bigg]
\end{align}
where the components of $\bm{U}^{\epsilon} _t$ are
\begin{align}
\label{Uequation}
(\bm{U}^{\epsilon} _t)_i & =  - \sum_j \big( ( \bm{\gamma}(\bm{X}^{\epsilon} _t) - \epsilon \bm{\kappa} (\bm{X}^{\epsilon} _t))^{-1} \big) _{ij} \epsilon (\bm{V}^{\epsilon} _t)_j + \sum_j \big( ( \bm{\gamma}(\bm{X} _0) - \epsilon \bm{\kappa} (\bm{X}_0))^{-1} \big) _{ij} \epsilon (\bm{V} ^{\epsilon} _0)_j \notag \\
& + \sum_{\ell , j} \Bigg[ \int_0^t \frac{\partial}{\partial X_{\ell}} \big[ \big( ( \bm{\gamma}(\bm{X}^{\epsilon} _s) - \epsilon \bm{\kappa} (\bm{X}^{\epsilon} _s))^{-1} \big) _{ij} \big] \; \times \notag \\
& \sum_{k_1 , k_2} \Big[ \int_0^{\infty} (e^{- \bm{\gamma}(\bm{X}^{\epsilon} _s)y})_{jk_1}(e^{- (\bm{\gamma} (\bm{X}^{\epsilon} _s ))^{\rm T} y})_{k_2 \ell} \; dy \; \times \notag \\
& \Big( [\epsilon \bm{V} ^{\epsilon} _s (\bm{F} (\bm{X}^{\epsilon} _s))^{\rm T}] _{k_1 k_2} ds \; + \; [\epsilon  \bm{V}^{\epsilon} _s ( \epsilon \bm{V}^{\epsilon} _s)^{\rm T} (\bm{\kappa} (\bm{X}^{\epsilon} _s))^{\rm T}] _{k_1 k_2} ds \notag \\
&+ \; [\epsilon \bm{V}^{\epsilon} _s (\bm{\sigma} d \bm{W} _s)^{\rm T} ] _{k_1 k_2} \; + \; [\bm{F} (\bm{X} ^{\epsilon} _s)(\epsilon \bm{V}^{\epsilon} _s)^{\rm T} ] _{k_1 k_2} ds \notag \\
&+ \; [ \bm{\kappa} (\bm{X}^{\epsilon} _s) \epsilon \bm{V}^{\epsilon} _s (\epsilon \bm{V}^{\epsilon} _s)^{\rm T} ]_{k_1 k_2} ds \; + \; [ \bm{\sigma} d \bm{W} _s (\epsilon \bm{V}^{\epsilon} _s)^{\rm T} ] _{k_1 k_2} \Big) \Big] \Bigg] \; .
\end{align}
We can now write equation~\eqref{afterbyparts} in the form of Lemma~1
\begin{equation*}
\bm{Y}^{\epsilon}_t = \bm{Y}_0 + \bm{U}^{\epsilon} _t + \int _0 ^t \bm{h}^{\epsilon} (\bm{Y}^{\epsilon} _s) d \bm{H}^{\epsilon} _s
\end{equation*}
by letting $\bm{h}^{\epsilon} : \mathbb{R}^{(m + 2n)} \rightarrow \mathbb{R}^{(m + 2n) \times (1 + n + 1 + (m + 2n)^2)}$ be the matrix-valued function given by
\begin{equation} \label{h epsilon}
\bm{h}^{\epsilon}(\bm{Y}) = \Big( ( \bm{\gamma}(\bm{Y}) - \epsilon \bm{\kappa} (\bm{Y}))^{-1} \bm{F}(\bm{Y}), ( \bm{\gamma}(\bm{Y}) - \epsilon \bm{\kappa} (\bm{Y}))^{-1} \bm{\sigma}, \bm{S}^{\epsilon} (\bm{Y}), \bm{\Lambda}^1 (\bm{Y}), \; ... \; , \bm{\Lambda}^{m + 2n} (\bm{Y}) \Big)
\end{equation}
where $\bm{S}^{\epsilon}: \mathbb{R}^{(m + 2n)} \rightarrow \mathbb{R}^{(m + 2n)}$ is the vector-valued function defined componentwise as
\begin{equation*}
S^{\epsilon}_i (\bm{Y}) = \sum_{\ell , j} \frac{\partial}{\partial Y_{\ell}} \big[ \big( ( \bm{\gamma}(\bm{Y}) - \epsilon \bm{\kappa} (\bm{Y}))^{-1} \big) _{ij} \big] J_{j \ell} (\bm{Y})
\end{equation*}
with $\bm{J}$ denoting the solution to equation~\eqref{Lyapunov}, and $\bm{\Lambda} ^{k_2} : \mathbb{R}^{(m + 2n)} \rightarrow \mathbb{R}^{(m + 2n) \times (m + 2n)}$ is defined componentwise as
\begin{equation*}
\Lambda ^{k_2} _{i k_1} (\bm{Y}) = \sum_{\ell , j} \frac{\partial}{\partial Y_{\ell}} \big[ \big( ( \bm{\gamma}(\bm{Y}) - \epsilon \bm{\kappa} (\bm{Y}))^{-1} \big) _{ij} \big] \left[ - \int _0 ^{\infty} (e^{- \bm{\gamma} (\bm{Y}) y})_{jk_1} (e^{- (\bm{\gamma} (\bm{Y}))^{\rm T} y})_{k_2 \ell} \; dy \right] ,
\end{equation*}
and by letting $\bm{H} ^{\epsilon}$ be the process, with paths in $C([0,T], \mathbb{R}^{1 + n + 1 + (m + 2n)^2})$, given by
\begin{equation} \label{H epsilon}
\bm{H}^{\epsilon} _t =
\left[\begin{array}{c}
t  \\
\bm{W}_t \\
t \\
(\epsilon \bm{V}^{\epsilon} _t)_1 \epsilon \bm{V}^{\epsilon} _t - ( \epsilon \bm{V}^{\epsilon} _0)_1 \epsilon \bm{V}^{\epsilon} _0 \\
\vdots \\
(\epsilon \bm{V}^{\epsilon} _t)_{(m + 2n)} \epsilon \bm{V}^{\epsilon} _t - ( \epsilon \bm{V}^{\epsilon} _0)_{(m + 2n)} \epsilon \bm{V}^{\epsilon} _0
\end{array}\right] .
\end{equation}

We now define 
\begin{equation} \label{h}
\bm{h} (\bm{Y}) = \Big( ( \bm{\gamma}(\bm{Y}))^{-1} \bm{F}(\bm{Y}), ( \bm{\gamma}(\bm{Y}))^{-1} \bm{\sigma}, \bm{S} (\bm{Y}), \bm{\Psi}^1 (\bm{Y}), \; ... \; , \bm{\Psi}^{m + 2n} (\bm{Y}) \Big)
\end{equation}
where $\bm{S}$ is defined componentwise as
\begin{equation*}
S_i (\bm{Y}) = \sum_{\ell , j} \frac{\partial}{\partial Y_{\ell}} \big[ \big( ( \bm{\gamma}(\bm{Y}) )^{-1} \big) _{ij} \big] J_{j \ell} (\bm{Y})
\end{equation*}
and $\bm{\Psi} ^{k_2}$ is defined componentwise as
\begin{equation*}
\Psi ^{k_2} _{i k_1} (\bm{Y}) = \sum_{\ell , j} \frac{\partial}{\partial Y_{\ell}} \big[ \big( ( \bm{\gamma}(\bm{Y}))^{-1} \big) _{ij} \big] \left[ - \int _0 ^{\infty} (e^{- \bm{\gamma} (\bm{Y}) y})_{jk_1} (e^{- (\bm{\gamma} (\bm{Y}))^{\rm T} y})_{k_2 \ell} \; dy \right] .
\end{equation*}
Letting
\begin{equation} \label{H}
\bm{H} _t =
\left[\begin{array}{c}
t  \\
\bm{W}_t \\
t \\
\bm{0} \\
\vdots \\
\bm{0}
\end{array}\right],
\end{equation}
we show in the next section that $\bm{U}^{\epsilon} $, $\bm{h}^{\epsilon}$, $\bm{H} ^{\epsilon}$, $\bm{h}$, and $\bm{H}$ satisfy the assumptions of Lemma~1.
It follows that, as $\epsilon~\rightarrow~0$, $\bm{X}^{\epsilon}$ converges to the solution of the equation
\begin{equation} \label{limiting equation 1}
d \bm{X} _t = \left[(\bm{\gamma} (\bm{X} _t))^{-1} \bm{F} (\bm{X} _t) + \bm{S} (\bm{X} _t) \right] dt + (\bm{\gamma} (\bm{X} _t))^{-1} \bm{\sigma} d \bm{W} _t \; .
\end{equation}
Letting $\bm{X} _t = (\bm{y}_t, \bm{\xi} _t, \bm{\zeta} _t)$ (i.e., analogously to $\bm{X}^{\epsilon} _t$, we let $\bm{y}_t$ stand for the vector of the first $m$ components of $\bm{X} _t$, $\bm{\xi} _t$ stand for the vector of the next $n$ components, and $\bm{\zeta}_t$ stand for the vector of the last $n$ components), we have
\begin{equation}
\label{gamma inverse}
(\bm{\gamma} (\bm{X} _t))^{-1} =
\left[
\setlength{\extrarowheight}{5pt}
\begin{array}{ccc}
(\bm{D}^1)^{-1} & \tilde{ \bm{g}}(\bm{y _t}) &  \frac{1}{\Gamma} \tilde{ \bm{g}}(\bm{y _t})  \\
\bm{0} & (\bm{D}^2) ^{-1} &  \frac{1}{\Gamma} (\bm{D}^2)^{-1} \\
\bm{0} &  - \frac{\Omega ^2}{\Gamma} (\bm{D} ^2)^{-1} &  \bm{0}
\end{array}
\setlength{\extrarowheight}{5pt}
\right]
\end{equation}
where $(\tilde{ \bm{g}} (\bm{y_t}))_{ij} = k_j g^{ij} (\bm{y_t})$.
Thus, from \eqref{limiting equation 1}, we obtain the following limiting equation for $\bm{y}$
\begin{align} \label{limiting equation 2}
d y^i _t = f^i (\bm{y} _t) dt + \sum _{p,j} g^{pj} (\bm{y} _t) \frac{\partial g^{ij} (\bm{y} _t)}{\partial y_p} & \left[ \frac{\frac{\Gamma}{\Omega ^2} \frac{\delta _p}{\tau _j} + \frac{1}{\Gamma} \left(1 - \frac{\delta _p}{\tau_j} \right)}{2 \left( \frac{\Gamma}{\Omega ^2} \frac{\delta _p}{\tau _j} \left(1 + \frac{\delta _p}{\tau _j} \right) + \frac{1}{\Gamma} \right) } \right] dt \\[1em]
&+ \sum_j g^{ij} (\bm{y}_t) dW^j _t \; . \notag
\end{align}  
Taking the limit $\Gamma, \Omega ^2 \rightarrow \infty$ while keeping $\frac{\Gamma}{\Omega ^2}$ constant, this becomes
\begin{equation}
d y^i _t = f^i (\bm{y} _t) dt + \sum _{p,j} g^{pj} (\bm{y} _t) \frac{\partial g^{ij} (\bm{y} _t)}{\partial y_p} {1\over2} \left( 1 + \frac{\delta _p}{\tau _j} \right)^{-1} dt + \sum_j g^{ij} (\bm{y}_t) dW^j _t \; .
\end{equation}  

Q.E.D.

\section{Verification of Conditions}

In this section we verify that the assumptions of Lemma~1 and the Conditions 1 and 2 in its statement are satisfied.  In order to do this, we will need the following lemmas.

\begin{lemma}
Let the functions $f^i$ and $g^{ij}$ satisfy the assumptions of Theorem~1.  Then there exist $\epsilon _0 > 0$ and $C > 0$ such that for $0 \leq \epsilon \leq \epsilon _0$, $\bm{\gamma}(\bm{X}) - \epsilon \bm{\kappa} (\bm{X})$ is invertible and $ \| ( \bm{\gamma}(\bm{X}) - \epsilon \bm{\kappa} (\bm{X}))^{-1} \| < C$ for all $\bm{X} \in \mathbb{R}^{m + 2n}$.  
\end{lemma}

\begin{proof}
Recall from (\ref{eigenvaluesgamma}) that the eigenvalues of $\bm{\gamma} (\bm{X})$ do not depend on $\bm{X}$ and are nonzero.  With this in mind, invertibility follows from the boundedness of $\bm{\kappa}$, the continuity of the function that maps a matrix to the vector of its eigenvalues (repeated according to their multiplicities), and the fact that, for fixed $\tilde{\epsilon} > 0$, the closure of the set $A_{\tilde{\epsilon}} = \{ \bm{\gamma} (\bm{X}) - \epsilon \bm{\kappa} (\bm{X})  : \bm{X} \in \mathbb{R}^{m + 2n} , \; 0 \leq \epsilon \leq \tilde{\epsilon} \}$ is compact since $\bm{\gamma}$ and $\bm{\kappa}$ are bounded.  The boundedness of the inverse follows from the compactness of the closure of $A_{\epsilon _0}$ and the fact that the map that takes a matrix to its inverse is a continuous function on the space of invertible matrices.
\end{proof}

\begin{lemma}
Let the functions $f^i$ and $g^{ij}$ satisfy the assumptions of Theorem~1 and let $\epsilon _0$ be as in Lemma 2.  Then there exists $C > 0$ such that for $0 \leq \epsilon \leq \epsilon _0$, $\bm{X} \in \mathbb{R}^{m + 2n}$, and $1 \leq \ell \leq m + 2n$,  
$$\left\| \frac{\partial}{\partial X_{\ell}} \big[ ( \bm{\gamma}(\bm{X}) - \epsilon \bm{\kappa} (\bm{X} ))^{-1} \big] \right\| < C \; . $$
\end{lemma}

\begin{proof}
Differentiating the identity $( \bm{\gamma}(\bm{X}) - \epsilon \bm{\kappa} (\bm{X}))^{-1}( \bm{\gamma}(\bm{X}) - \epsilon \bm{\kappa} (\bm{X})) = \bm{I}$, we obtain
\begin{align} \label{derivative of inverse}
 \frac{\partial}{\partial X_{\ell}} \big[ ( &\bm{\gamma}(\bm{X}) - \epsilon \bm{\kappa} (\bm{X}))^{-1} \big] =\\
 &- ( \bm{\gamma}(\bm{X}) - \epsilon \bm{\kappa} (\bm{X}))^{-1} \left[ \frac{\partial}{\partial X_{\ell}} \big[ \bm{\gamma}(\bm{X}) - \epsilon \bm{\kappa} (\bm{X}) \big] \right] ( \bm{\gamma}(\bm{X}) - \epsilon \bm{\kappa} (\bm{X}))^{-1} \; . \notag
\end{align}
From the assumption that the derivatives of the $g^{ij}$ and the second derivatives of the $f^i$ are bounded, it follows that $\frac{\partial \bm{\gamma}}{\partial X_{\ell}}$ and $ \frac{\partial \bm{\kappa}}{\partial X_{\ell}}$ are bounded functions of $\bm{X}$.  The statement then follows from this observation, Lemma 2, and equation (\ref{derivative of inverse}).
\end{proof}

We introduce some notation that will be used in the following.  Let $\bm{\Phi} _{t_0} (t)$ be the fundamental solution matrix of the constant coefficient system
\begin{equation}
\label{Phiequation}
\frac{d}{dt} \bm{\Phi} (t)  = - \frac{1}{\epsilon} \bm{M} \bm{\Phi} (t)
\end{equation}
satisfying $\bm{\Phi} _{t_0} (t_0) = \bm{I}$,
where
\begin{equation}
\label{Aequation}
\bm{M} = \left[
\setlength{\extrarowheight}{5pt}
\begin{array}{cc}
\bm{0} & - \frac{\Gamma}{\Omega ^2} \bm{D} ^2 \\
\Gamma \bm{D} ^2 &  \frac{\Gamma ^2}{\Omega ^2} \bm{D}^2
\end{array}
\setlength{\extrarowheight}{5pt}
\right].
\end{equation}
Similarly, let $\bm{\psi} _{t_0} (t)$ be the fundamental solution matrix of the variable coefficient system
\begin{equation}
\label{psiequation}
\frac{d}{dt} \bm{\psi} (t) =  \left( - \frac{1}{\epsilon} \bm{D^1} + \hat{\bm{J}}_f (\bm{y}^{\epsilon} _t) \right) \bm{\psi} (t)
\end{equation}
satisfying $\bm{\psi}_{t_0} (t_0) = {\bm I}$.

\begin{lemma}
For each $\epsilon > 0$, let $\bm{y}^{\epsilon}$ be any process with paths in $C([0,T], \mathbb{R}^{m})$
and let $\bm{\Phi} _{t_0} (t)$ and  $\bm{\psi} _{t_0} (t)$ be defined as above.  Let the functions $f^i$ satisfy the assumptions of Theorem~1.  Then there exist $C$, $C_d > 0$ independent of $\epsilon$ such that for $0 \leq t_0 \leq t \leq T$,
\begin{equation}
\label{lemma3statement}
\| \bm{\Phi} _{t_0} (t)  \| \leq  C \exp \left\lbrace - \frac{ C_d (t - t_0)}{\epsilon} \right\rbrace 
\end{equation}
and
\begin{equation}
\label{lemma3statement}
\| \bm{\psi} _{t_0} (t)  \| \leq  C \exp \left\lbrace - \frac{ C_d (t - t_0)}{\epsilon} \right\rbrace  .
\end{equation}

\end{lemma}

\begin{proof}

Let $\bm{\beta}$ be a vector-valued function which solves 

\begin{equation}
\label{betaequation}
\frac{d}{dt} \bm{\beta} (t)
 = - \frac{1}{\epsilon} \bm{M} \bm{\beta} (t)
\end{equation}
where $\bm{M}$ is defined in (\ref{Aequation}).
The eigenvalues $\lambda_1, \lambda_2, ..., \lambda_{2n}$ of $\bm{M}$ are equal to
\begin{equation*}
\frac{\Gamma ^2}{2k_j \Omega ^2} \left[1 \pm \sqrt{1 - 4 \frac{\Omega ^2}{\Gamma ^2}} \right], \hspace{5pt}  j = 1, ..., n
\end{equation*}
and $\bm{M}$ is diagonalizable if $\Gamma ^2 \neq 4 \Omega ^2$ (if $\Gamma ^2 = 4 \Omega ^2$, an argument similar to the one below follows using the Jordan form of $\bm{M}$).  Writing $\bm{M} = \bm{P} \bm{\Lambda} \bm{P}^{-1}$, where $\bm{\Lambda}$ is the diagonal matrix consisting of $\lambda_1, \lambda_2, ..., \lambda_{2n}$, gives
\begin{equation*}
\bm{\beta} (t) = \bm{P} \left[\begin{array}{cccc}
e^{\frac{-(t - t_0) \lambda _1}{\epsilon}} & 0 & ... & 0  \\
0 & e^{\frac{-(t - t_0) \lambda _2}{\epsilon}} & ... & 0\\
\vdots & \vdots & \ddots & \vdots \\
0 & 0 & ... & e^{\frac{-(t - t_0) \lambda _{2n}}{\epsilon}}
\end{array}\right] \bm{P} ^{-1} \bm{\beta} (t_0) \; .
\end{equation*}
Let $c_{\lambda} = \min_{1 \leq j \leq 2n} Re (\lambda _j) > 0$.  Then we have
\begin{equation}
\label{munubound}
\| \bm{\beta} (t) \| \leq C_1 \| \bm{\beta} (t_0) \| e^{\frac{- c_{\lambda} (t - t_0)}{\epsilon}} 
\end{equation}
where $C_1$ is a constant.
Now, let $(\bm{\Phi}_{t_0} (t))_{\cdot j}$ denote the $j^{th}$ column of $\bm{\Phi}_{t_0} (t)$.  Then, since $(\bm{\Phi}_{t_0} (t))_{\cdot j}$ solves (\ref{betaequation}), by (\ref{munubound}) and by the chain of inequalities
\begin{equation}
\label{chain}
\| \bm{\Phi} _{t_0} (t) \| \leq C_2 \| \bm{\Phi} _{t_0} (t)  \| _1  = C_2 \max _j  \| (\bm{\Phi} _{t_0} (t))_{\cdot j}  \| _1 \leq C_3 \max _j  \| (\bm{\Phi} _{t_0} (t))_{\cdot j}  \| \; ,
\end{equation}
where  $\| \cdot \|_1$ denotes the vector $l^1$ norm  or the induced matrix $l^1$ norm depending on its argument, and $C_2$ and $C_3$ are constants, we have

$$\| \bm{\Phi}_{t_0} (t) \| \leq C e^{\frac{- c_{\lambda} (t - t_0)}{\epsilon}}$$
for $0 \leq t_0 \leq t \leq T$.

Next, let $\bm{u}$ be a process with paths in $C([0,T], \mathbb{R}^{m})$ that solves the equation 
$$\frac{d}{dt} \bm{u} (t) =  \left(- \frac{1}{\epsilon} \bm{D^1} + \hat{\bm{J}}_f (\bm{y}^{\epsilon} _t) \right) \bm{u} (t) \; .
$$
Then
\begin{align*}
\frac{d}{dt} \Big( \| \bm{u}(t) \| ^2 \Big) &= \frac{d}{dt} \Big( \bm{u}(t) ^T \bm{u}(t) \Big) \\
&= 2 \Bigg( \Big( - \frac{1}{\epsilon} \bm{D^1} + \hat{\bm{J}}_f (\bm{y}^{\epsilon} _t) \Big) \bm{u} (t) \Bigg)^T \bm{u}(t) \\
& \leq \frac{- 2}{ c \epsilon} \| \bm{u}(t) \| ^2 + 2 \| \hat{\bm{J}}_f (\bm{y}^{\epsilon} _t)  \bm{u}(t) \| \| \bm{u}(t) \|\\
& \leq 2 \left( \frac{- 1}{ c \epsilon}  + C_4 \right)  \| \bm{u}(t) \| ^2
\end{align*}
where $c = \max_{1 \leq i \leq m} \; c_i > 0$ (recall that $\bm{D}^1$ is the diagonal matrix with entries $\frac{1}{c_i}$) and $C_4$ is a constant that bounds $\| \hat{\bm{J}}_f (\bm{y}^{\epsilon} _t) \|$ (such a bound exists by the assumption that the first derivatives of the $f^i$ are bounded).  Thus, by Gronwall's inequality, we have
$$\| \bm{u} (t) \| ^2 \leq \| \bm{u}(t_0) \| ^2 e^{2( \frac{-1}{ c \epsilon}  + C_4) (t - t_0)} ,$$
so that
$$\| \bm{u} (t) \| \leq C_5 \| \bm{u}(t_0) \|  e^{\frac{- (t - t_0)}{ c \epsilon}} $$
for $0 \leq t_0 \leq t \leq T$, where $C_5$ depends on $T$.  
Then, by the analogue of (\ref{chain}) for $\bm{\psi}_{t_0} (t)$, we have

$$\| \bm{\psi}_{t_0} (t) \| \leq C e^{\frac{- (t - t_0)}{ c \epsilon}} .$$

\end{proof}

\begin{lemma}
Let $\bm{K} \in \mathbb{R} ^{2n \times n}$ be a constant, nonrandom matrix.  Then there exists $C>0$ independent of $\epsilon$ such that 
\begin{equation*}
 E \left[ \left( \sup_{0 \leq t \leq T} \left\| \int _0 ^t \bm{\Phi}_s (t)  \bm{K} d \bm{W} _s \right\| \right) ^2 \right] \leq C \epsilon ^{1 / 2} \; .
\end{equation*}
\end{lemma}

\begin{proof}
We begin by following the first part of the argument of \cite[Lemma 3.7]{Pavliotis}.  We fix $\alpha \in (0, \frac{1}{2})$ and use the factorization method from \cite{Da Prato} (see also \cite[sec. 5.3]{Da Prato 2}) to rewrite
\begin{align*}
\bm{I}(t) &= \int _0 ^t \bm{\Phi}_s (t) \bm{K} d \bm{W} _s \\
 &= \frac{\sin(\pi \alpha)}{\pi} \int _0 ^t \bm{\Phi}_s (t) (t - s)^{\alpha - 1} \bm{Y}(s) ds \; ,
\end{align*}
where 
$$\bm{Y}(s) = \int_0^s \bm{\Phi}_u (s) (s - u)^{- \alpha} \bm{K} d \bm{W} _u \; .$$
This identity follows from the property of fundamental solution matrices $\bm{\Phi}_s (t) \bm{\Phi}_u (s) = \bm{\Phi}_u (t)$ \cite{Hartman} and the identity
$$\int _u ^t (t - s) ^{\alpha - 1} (s - u)^{- \alpha} ds = \frac{\pi}{\sin(\pi \alpha)} \; , \hspace{10pt} 0 < \alpha < 1 \; .$$

We fix $m > \frac{1}{2 \alpha}$ and use the H$\ddot{\mathrm{o}}$lder inequality:
\begin{equation*}
\| \bm{I}(t) \| ^{2m} \leq C_1 \left( \int _0 ^t \| \bm{\Phi}_s (t) (t - s)^{\alpha - 1} \| ^{\frac{2m}{2m - 1}} ds \right) ^{2m - 1} \int_0 ^t \| \bm{Y}(s) \| ^{2m} ds \; .
\end{equation*}
Using Lemma 4 and the change of variables $z = \frac{2m}{2m-1} \frac{C_d (t-s)}{\epsilon} $, we have
\begin{align*}
\int _0 ^t \| \bm{\Phi}_s (t) (t - s)^{\alpha - 1} \| ^{\frac{2m}{2m - 1}} ds
& \leq C_2 \int _0 ^t e^{- \frac{2m}{2m - 1} \frac{C_d (t - s)}{\epsilon}}  (t - s)^{\frac{2m}{2m - 1} ( \alpha - 1)}  ds \\
& = C_2 \left(\frac{2m - 1}{2 C_d m} \right) ^{\frac{2m \alpha - 1}{2m - 1}} \epsilon ^{\frac{2m \alpha - 1}{2m - 1}} \int_0 ^{\frac{C_d t}{\epsilon} \frac{2 m}{2m - 1}} e ^{- z} z ^{\frac{2m}{2m - 1} (\alpha - 1)} dz \\
& \leq C_3 \epsilon ^{\frac{2m \alpha - 1}{2m - 1}} \int_0 ^{\infty} e ^{- z} z ^{\frac{2m}{2m - 1} (\alpha - 1)} dz \\
& \leq C_4 \epsilon ^{\frac{2m \alpha - 1}{2m - 1}}
\end{align*}
where in the above we have used the fact that $e ^{- z} z ^{\frac{2m}{2m - 1} (\alpha - 1)} \in L^1 (\mathbb{R} ^+)$ since $m > \frac{1}{2 \alpha}$.  Therefore, we have
\begin{equation*}
E \left[ \sup_{0 \leq t \leq T} \| \bm{I}(t) \| ^{2m} \right] \leq C_5 \epsilon ^{2m \alpha - 1} E \left[ \int _0^T \| \bm{Y}(s) \| ^{2m} ds \right] .
\end{equation*}

In the following, for a matrix $\bm{A}$, let $\| \bm{A} \| _{HS} = \sqrt{\sum _{i,j} A_{ij} ^2}$ denote the Hilbert-Schmidt norm of $\bm{A}$.  Then, letting $\frac{1}{4} < \alpha < \frac{1}{2}$ and $m = 2$, we have
\begin{align*}
E \left[ \| \bm{Y}(t) \| ^{4} \right] & \leq 2n \sum _{i = 1} ^{2n} E\left[ \big( Y _i (t) \big) ^4  \right] \hspace{20pt} \textrm{by the Cauchy-Schwarz inequality} \\
& = 6n \sum _{i = 1} ^{2n} \left( E \left[ \big( Y _i (t) \big) ^2  \right] \right) ^2 \hspace{20pt} \textrm{since, for each $i$, $Y_i (t)$ is Gaussian}\\
& \leq 6n \bigg( E \Big[ \| \bm{Y}(t) \| ^2 \Big] \bigg) ^2 \\
& = 6n \left( \int_0^t \| \bm{\Phi}_u (t) (t - u)^{- \alpha} \bm{K} \| _{HS} ^2 \; du \right) ^2 \\
& \hspace{30pt} \textrm{by the It\^o isometry (see \cite[Theorem (4.4.14)]{Arnold})} \\
& \leq C_6 \left( \int_0^t \| \bm{\Phi}_u (t)  \| ^2 (t - u)^{- 2 \alpha} du \right) ^2 \\
& \leq C_7 \left( \int_0^t e^{- \frac{2 C_d (t -u)}{\epsilon}} (t - u)^{- 2 \alpha} du \right) ^2 \hspace{20pt} \textrm{by Lemma 4} \\
& = C_7 \left( \frac{\epsilon}{2 C_d} \right) ^{2(1 - 2 \alpha)} \left( \int_0 ^{\frac{2 C_d t}{\epsilon}} e^{-s} s^{-2 \alpha} ds \right) ^2 \\
& \leq C_8 \epsilon ^{2(1 - 2 \alpha)} \hspace{20pt} \textrm{using the fact that $e^{-s} s^{-2 \alpha} \in L^1 (\mathbb{R} ^+)$ since $\alpha < \frac{1}{2}$} \; .  
\end{align*}

Thus,
\begin{equation*}
E \left[ \sup_{0 \leq t \leq T} \| \bm{I}(t) \| ^{4} \right] \leq C_9 \epsilon \; , 
\end{equation*}
where $C_9$ is a constant that depends on $T$, so by the Cauchy-Schwarz inequality,
\begin{equation*}
E \left[ \sup_{0 \leq t \leq T} \| \bm{I}(t) \| ^{2} \right] \leq \left( E \left[ \sup_{0 \leq t \leq T} \| \bm{I}(t) \| ^{4} \right] \right) ^{1 / 2} \leq C \epsilon ^{1 / 2} \; . 
\end{equation*}

\end{proof}

\begin{lemma}
For each $\epsilon > 0$, let $\bm{X}^{\epsilon}$ be any process with paths in $C([0,T], \mathbb{R}^{m + 2n})$ and let $\bm{V}^{\epsilon}$ be the solution to the second equation in~\eqref{vector form}, where the functions $f^i$ and $g^{ij}$ satisfy the assumptions of Theorem~1, with the initial condition \newline $\bm{V} ^{\epsilon} _0 =(\bm{v}_0, \bm{\eta} ^{\epsilon} _0, \bm{z} ^{\epsilon} _0)$ defined in the statement of Theorem~1.  Then, as $\epsilon \rightarrow 0$, $\epsilon \bm{V}^{\epsilon} \rightarrow 0$ in $L^2$, and therefore in probability, with respect to $C([0,T], \mathbb{R}^{m + 2n})$, i.e.
\begin{equation}
\label{Lemma6result}
\lim _{\epsilon \rightarrow 0} E \left[ \left( \sup_{0 \leq t \leq T} \| \epsilon \bm{V}^{\epsilon} _t \| \right) ^2 \right] = 0
\end{equation}
and so, for all $a > 0$,
$$ \lim _{\epsilon \rightarrow 0} P \left( \sup_{0 \leq t \leq T} \| \epsilon \bm{V}^{\epsilon} _t \| > a \right)   = 0 \; .$$
Note that adaptedness of  $\bm{X}^{\epsilon}$ is not required.
\end{lemma}

\begin{proof}

We first consider the vector $(\bm{\eta}^{\epsilon} , \bm{z}^{\epsilon} )$, which consists of the last $2n$ components of $\bm{V}^{\epsilon}$, and show that $\epsilon(\bm{\eta}^{\epsilon} , \bm{z}^{\epsilon} )$ goes to zero in $L^2$ with respect to $C([0,T], \mathbb{R}^{2n})$.  We solve the second equation in~\eqref{vector form} for $(\bm{\eta}^{\epsilon}_t, \bm{z}^{\epsilon}_t)$.  The equation for $(\bm{\eta}^{\epsilon}_t, \bm{z}^{\epsilon}_t)$ is a linear SDE so its solution is \cite{Arnold}
\begin{equation*}
\left[\begin{array}{c}
\bm{\eta}^{\epsilon}_t \\
\bm{z}^{\epsilon}_t \\
\end{array}\right]
= \bm{\Phi} _0 (t) 
\left[\begin{array}{c}
\bm{\eta}^{\epsilon}_0 \\
\bm{z}^{\epsilon}_0 \\
\end{array}\right]
 + \frac{1}{\epsilon} \int _0 ^t \bm{\Phi} _0 (t) (\bm{\Phi} _0 (s))^{-1} 
\left[\begin{array}{c}
\bm{0} \\
\Gamma \bm{D}^2
\end{array}\right] 
 d \bm{W} _s
\end{equation*}
where $\bm{\Phi} _0 (t)$ is the fundamental solution matrix of the equation
\begin{equation*}
\frac{d}{dt} \bm{\Phi} (t)  = - \frac{1}{\epsilon} \bm{M} \bm{\Phi} (t)
\end{equation*}
satisfying $\bm{\Phi} _0 (0) = {\bm I}$, where $\bm{M}$ is defined in (\ref{Aequation}).
Then, using $\bm{\Phi} _0 (t) = \bm{\Phi}_s (t) \bm{\Phi} _0 (s)$ and the Cauchy-Schwarz inequality, we get
\begin{align*}
E \left[ \left( \sup_{0 \leq t \leq T} \left\| \epsilon \left[\begin{array}{c}
\bm{\eta}^{\epsilon}_t \\
\bm{z}^{\epsilon}_t \\
\end{array}\right] \right\| \right) ^2 \right] &\leq 2E \left[ \left( \sup_{0 \leq t \leq T} \left\| \epsilon \bm{\Phi} _0 (t) 
\left[\begin{array}{c}
\bm{\eta}^{\epsilon}_0 \\
\bm{z}^{\epsilon}_0 \\
\end{array}\right] \right\| \right) ^2 \right] \\
 &+ 2 E \left[ \left( \sup_{0 \leq t \leq T} \left\| \int _0 ^t \bm{\Phi} _s (t)
\left[\begin{array}{c}
\bm{0} \\
\Gamma \bm{D}^2
\end{array}\right] 
 d \bm{W} _s \right\| \right) ^2 \right] \; .
\end{align*}
Now, 
\begin{equation*}
E \left[ \left( \sup_{0 \leq t \leq T} \left\| \epsilon \bm{\Phi} _0 (t) 
\left[\begin{array}{c}
\bm{\eta}^{\epsilon}_0 \\
\bm{z}^{\epsilon}_0 \\
\end{array}\right] \right\| \right) ^2 \right] \leq 
\epsilon ^2 \left( \sup_{0 \leq t \leq T} \left\| \bm{\Phi} _0 (t) 
 \right\| \right) ^2 E \left[ 
 \left\| \left[\begin{array}{c}
\bm{\eta}^{\epsilon}_0 \\
\bm{z}^{\epsilon}_0 \\
\end{array}\right] \right\| ^2 \right]  \leq C_1 \epsilon
\end{equation*}
where the last inequality follows from Lemma 4 and (\ref{propertiesHNstat}) in the Appendix (recall that $(\bm{\eta}^{\epsilon}_0, \bm{z}^{\epsilon}_0)$ is distributed according to the stationary distribution corresponding to \eqref{eq:harmonicnoise}).  Thus, using this bound and Lemma 5, we have, for $0 < \epsilon < 1$,
\begin{equation}
\label{etaandzbound}
E \left[ \left( \sup_{0 \leq t \leq T} \left\| \epsilon \left[\begin{array}{c}
\bm{\eta}^{\epsilon}_t \\
\bm{z}^{\epsilon}_t \\
\end{array}\right] \right\| \right) ^2 \right] \leq C_2 \epsilon ^{1 / 2} \; .
\end{equation}

Next we consider the vector $\bm{v}^{\epsilon}$, which consists of the first $m$ components of $\bm{V}^{\epsilon}$, and show that $\epsilon \bm{v}^{\epsilon}$ goes to zero in $L^2$ with respect to $C([0,T], \mathbb{R}^{m})$. Solving the second equation in~\eqref{vector form} for $\bm{v}^{\epsilon}_t$ gives
\begin{equation*}
\bm{v}^{\epsilon}_t = \bm{\psi} _0 (t) \bm{v}_0 + \frac{1}{\epsilon}  \int _0 ^t \bm{\psi} _0 (t) (\bm{\psi} _0 (s))^{-1} \hat{\bm{f}}( \bm{y} ^{\epsilon} _s) ds + \frac{1}{\epsilon} \int_{0}^t \bm{\psi} _0 (t) (\bm{\psi} _0 (s))^{-1}  \hat{\bm{g}}(\bm{y}^{\epsilon} _s) \bm{\eta}^{\epsilon}_s ds
\end{equation*}
where $\bm{\psi} _0 (t)$ is the fundamental solution matrix of the equation
\begin{equation*}
\frac{d}{dt} \bm{\psi} (t) =  \left( - \frac{1}{\epsilon} \bm{D}^1 + \hat{\bm{J}}_f (\bm{y}^{\epsilon} _t) \right) \bm{\psi} (t)
\end{equation*}
satisfying $\bm{\psi} _0 (0) = {\bm I}$.  We use $\bm{\psi} _0 (t) = \bm{\psi}_{s} (t)\bm{\psi} _0 (s)$, Lemma 4, and the boundedness of $\hat{\bm{f}}$ to get
\begin{align}
\label{boundonintegralf}
\sup _{0 \leq t \leq T} \Bigg\|  \int _0 ^t \bm{\psi} _0 (t) (\bm{\psi} _0 (s))^{-1} \hat{\bm{f}}( \bm{y} ^{\epsilon} _s)  ds \Bigg\| & \leq  \sup _{0 \leq t \leq T} C_3 \int _0 ^t e^ { - \frac{ C_d (t - s)}{\epsilon} }  ds \notag \\
& \leq  C_3 \int _0 ^T  e ^{ - \frac{ C_d (T - s)}{\epsilon} }  ds \notag \\
& = \frac{C_3}{C_d} \epsilon \int _0 ^{\frac{C_d T}{\epsilon}}  e^{-u} du \; \leq \;  C_4  \epsilon \; .
\end{align}
Next, using Lemma 4, the boundedness of $\hat{\bm{g}}$, and (\ref{etaandzbound}), we have
\begin{align}
\label{boundonintegralg}
&E \left[ \left( \sup_{0 \leq t \leq T} \left\| \int_{0}^t \bm{\psi} _0 (t) (\bm{\psi} _0 (s))^{-1} \hat{\bm{g}}(\bm{y}^{\epsilon} _s) \bm{\eta}^{\epsilon}_s ds \right\| \right) ^2 \right] \notag \\
& \hspace{140pt} \leq C_5 E \left[ \left( \sup_{0 \leq t \leq T} \| \bm{\eta}^{\epsilon}_t \|   \right) ^2 \right] \left( \int_{0}^T e^{\frac{- C_d (T - s)}{\epsilon}} ds \right) ^2 \notag \\
& \hspace{140pt} \leq C_6 E\left[ \left( \sup_{0 \leq t \leq T} \| \epsilon \bm{\eta}^{\epsilon}_t \|   \right) ^2 \right] \; \leq \; C_7 \epsilon ^{1 / 2} \; .
\end{align}
Thus, by the Cauchy-Schwarz inequality, Lemma 4, (\ref{boundonintegralf}), and (\ref{boundonintegralg}),
\begin{align}
\label{bound on expectation}
E \left[ \sup _{0 \leq t \leq T} \| \epsilon \bm{v}^{\epsilon} _t \| ^2 \right] & \leq 3 E \left[ \sup_{0 \leq t \leq T} \| \epsilon \bm{\psi} _0 (t)  \bm{v}_0 \| ^2 \right] \notag \\
& + 3 E \left[ \sup _{0 \leq t \leq T} \left\| \int _0 ^t \bm{\psi} _0 (t) (\bm{\psi} _0 (s))^{-1} \hat{\bm{f}}( \bm{y}^{\epsilon} _s) ds \right\|^2 \right] \notag \\
& + 3 E \left[ \sup _{0 \leq t \leq T} \left\| \int _0 ^t \bm{\psi} _0 (t) (\bm{\psi} _0 (s))^{-1} \hat{\bm{g}}(\bm{y}^{\epsilon} _s) \bm{\eta}^{\epsilon}_s ds \right\| ^2 \right] \leq C_8 \epsilon ^{1 / 2} \; .
\end{align}
Thus, from (\ref{etaandzbound}) and (\ref{bound on expectation}) we have
\begin{equation}
\label{lastboundinLemma6}
E \left[ \left( \sup_{0 \leq t \leq T} \| \epsilon \bm{V}^{\epsilon} _t \| \right) ^2 \right] \leq C \epsilon ^{1 / 2}
\end{equation}
from which (\ref{Lemma6result}) follows.  The second claim then follows from (\ref{Lemma6result}) and Chebyshev's inequality.

\end{proof}

\begin{lemma}

For each $\epsilon > 0$, let $\bm{X}^{\epsilon}$ be any $\mathcal{F}_t$-adapted process with paths in $C([0,T], \mathbb{R}^{m + 2n})$ and let $\bm{V}^{\epsilon}$ again be the solution to the second equation in~\eqref{vector form}, where the functions $f^i$ and $g^{ij}$ satisfy the assumptions of Theorem~1, with the initial condition $\bm{V} ^{\epsilon} _0$ defined in the statement of Theorem~1.  Let $g: \mathbb{R}^{m + 2n} \rightarrow \mathbb{R}$ be a continuous and bounded function.  Then
\begin{equation}
\label{first equation in lemma}
\lim_{\epsilon \rightarrow 0} E \left[ \left( \sup_{0 \leq t \leq T} \left| \int _0^t g(\bm{X} ^{\epsilon} _s) \epsilon (\bm{V} ^ {\epsilon} _s) _i ds \right| \right) ^2 \right] = 0 \; ,
\end{equation}
\begin{equation}
\label{second equation in lemma}
\lim_{\epsilon \rightarrow 0} E \left[ \sup_{0 \leq t \leq T} \left| \int _0^t g(\bm{X} ^{\epsilon} _s) \epsilon (\bm{V} ^ {\epsilon} _s) _i \epsilon (\bm{V} ^ {\epsilon} _s) _{\ell} ds \right| \right] = 0 \; ,
\end{equation}
and
\begin{equation}
\label{third equation in lemma}
\lim_{\epsilon \rightarrow 0} E \left[ \left( \sup_{0 \leq t \leq T} \left| \int _0^t g(\bm{X} ^{\epsilon} _s) \epsilon (\bm{V} ^ {\epsilon} _s) _i d W ^j _s \right| \right) ^2 \right] = 0
\end{equation}
for all $i, \ell = 1, ..., m + 2n$ and $j = 1, ..., n$.
\end{lemma}

\begin{proof}
First, using the Cauchy-Schwarz inequality,
\begin{align*}
E \left[ \left( \sup_{0 \leq t \leq T} \left| \int _0^t g(\bm{X} ^{\epsilon} _s) \epsilon (\bm{V} ^ {\epsilon} _s) _i ds \right| \right) ^2 \right] & \leq E \left[ \left(  \int _0^T \left| g(\bm{X} ^{\epsilon} _s) \epsilon (\bm{V} ^ {\epsilon} _s) _i \right| ds  \right) ^2 \right] \\
& \leq  T  \int _0^T E \Big[ \Big( g(\bm{X} ^{\epsilon} _s) \epsilon (\bm{V} ^ {\epsilon} _s) _i \Big) ^2 \Big] ds  \\
& \leq C ^2 T \int _0^T E \Big[ \Big( \epsilon (\bm{V} ^ {\epsilon} _s) _i \Big) ^2 \Big] ds
\end{align*}
where $C$ is a constant that bounds $g$.  Then, using (\ref{lastboundinLemma6}), we get (\ref{first equation in lemma}).
Next, using the Cauchy-Schwarz inequality,
\begin{align*}
E \left[ \sup_{0 \leq t \leq T} \left| \int _0^t g(\bm{X} ^{\epsilon} _s) \epsilon (\bm{V} ^ {\epsilon} _s) _i \epsilon (\bm{V} ^ {\epsilon} _s) _{\ell} ds \right| \right] & \leq C \int _0^T E \Big[ \big| \epsilon (\bm{V} ^ {\epsilon} _s) _i \epsilon (\bm{V} ^ {\epsilon} _s) _{\ell} \big| \Big] ds \\
& \leq C \int _0^T \left( E \Big[ \Big( \epsilon (\bm{V} ^ {\epsilon} _s) _i \Big) ^2 \Big] E \Big[ \Big( \epsilon (\bm{V} ^ {\epsilon} _s) _{\ell} \Big) ^2 \Big] \right) ^{1 / 2} ds \; ,
\end{align*}
which gives (\ref{second equation in lemma}) by again using (\ref{lastboundinLemma6}).  Finally, for the It\^o integral, we first use Doob's maximal inequality and then use the It\^o isometry:
\begin{align*}
E \left[ \left( \sup_{0 \leq t \leq T} \left| \int _0^t g(\bm{X} ^{\epsilon} _s) \epsilon (\bm{V} ^ {\epsilon} _s) _i d W ^j _s \right| \right) ^2 \right] & \leq 4 E \left[ \left(  \int _0^T  g(\bm{X} ^{\epsilon} _s) \epsilon (\bm{V} ^ {\epsilon} _s) _i  d W ^j _s  \right) ^2 \right] \\
& = 4    \int _0^T E \Big[ \Big(  g(\bm{X} ^{\epsilon} _s) \epsilon (\bm{V} ^ {\epsilon} _s) _i \Big) ^2 \Big]  ds  \\
& \leq 4 C^2 \int _0^T E \Big[ \Big( \epsilon (\bm{V} ^ {\epsilon} _s) _i \Big) ^2 \Big]  ds \; ,
\end{align*}
from which (\ref{third equation in lemma}) follows by using (\ref{lastboundinLemma6}) one more time.
\end{proof}

For a fixed $t$, we will need a stronger bound on the rate of convergence of $E \left[ \| \epsilon \bm{V}^{\epsilon} _t \| ^2  \right]$ to zero than the one in (\ref{lastboundinLemma6}).  Such a bound is the content of the following lemma.
\begin{lemma}
For each $\epsilon > 0$, let $\bm{X}^{\epsilon}$ be any process with paths in $C([0,T], \mathbb{R}^{m + 2n})$ and let $\bm{V}^{\epsilon}$ again be the solution to the second equation in~\eqref{vector form}, where the functions $f^i$ and $g^{ij}$ satisfy the assumptions of Theorem~1, with the initial condition $\bm{V} ^{\epsilon} _0$ defined in the statement of Theorem~1.  Then there exists a constant $C$ independent of $\epsilon$ such that for $0 \leq t \leq T$,
\begin{equation}
\label{boundonevsquared}
E \left[ \epsilon \|  \bm{V}^{\epsilon} _t \| ^2  \right] \leq C \; .
\end{equation}

\end{lemma}

\begin{proof}
Let $\bm{K} = \left[\begin{array}{c}
\bm{0} \\
\Gamma \bm{D}^2
\end{array}\right] $ and let again $\| \cdot \| _{HS}$ denote the Hilbert-Schmidt norm.  Then, using the It\^o isometry and Lemma 4, 
\begin{align*}
E \left[ \left\| \int _0 ^t \bm{\Phi} _0 (t) (\bm{\Phi} _0 (s))^{-1} \bm{K} d \bm{W} _s \right\| ^2 \right] &= \int _0 ^t E \Big[ \| \bm{\Phi} _0 (t) (\bm{\Phi} _0 (s))^{-1} \bm{K} \| _{HS} ^2 \Big] ds \\
& \leq C_1 \int _0 ^t E \Big[ \| \bm{\Phi} _0 (t) (\bm{\Phi} _0 (s))^{-1} \bm{K} \| ^2 \Big] ds \\
& \leq C_2 \int _0 ^t  e^{- \frac{2 C_d (t - s)}{\epsilon}} ds \\
& \leq C \epsilon \; .
\end{align*}
Using this bound and the inequalities in the proof of Lemma 6 (without supremum over $t$), we get (\ref{boundonevsquared}).

\end{proof}

We are now ready to show that the assumptions of Lemma~1 and the Conditions 1 and 2 in its statement are satisfied.  We first show that the assumption \eqref{KP assumption} holds, where $\bm{U}^{\epsilon}$, $\bm{H}^{\epsilon}$, and $\bm{H}$ are defined in equations~\eqref{Uequation}, \eqref{H epsilon}, and \eqref{H} respectively.  The fact that $\bm{H}^{\epsilon} \rightarrow \bm{H}$ in probability with respect to $C([0,T], \mathbb{R}^{1 + n + 1 + (m + 2n)^2})$ is an immediate consequence of Lemma~6.  To see that $\bm{U}^{\epsilon}$ converges to zero in probability with respect to $C([0,T], \mathbb{R}^{m + 2n})$, observe that $\int_0^{\infty} (e^{- \bm{\gamma}(\bm{X}^{\epsilon} _s)y})_{jk_1}(e^{- (\bm{\gamma} (\bm{X}^{\epsilon} _s ))^{\rm T} y})_{k_2 \ell} \; dy$ is a bounded function of $\bm{X}^{\epsilon} _s$ since the eigenvalues of $\bm{\gamma} (\bm{X}^{\epsilon} _s )$ are independent of the value of $\bm{X}^{\epsilon} _s$ and have positive real parts.  With this in mind, the claim follows from Lemmas~2, 3, 6, and 7.

We now check Condition~1 of Lemma~1.  To do this, we find the Doob-Meyer decomposition of $\bm{H}^{\epsilon}$, i.e. the decomposition $\bm{H}^{\epsilon} = \bm{M}^{\epsilon} + \bm{A}^{\epsilon}$ where $\bm{M}^{\epsilon}$ is a local martingale and $\bm{A}^{\epsilon}$ is a process of locally bounded variation.  First, note that the columns of the matrix $ \epsilon \bm{V}^{\epsilon} _t (\epsilon \bm{V}^{\epsilon} _t)^{\rm T} - \epsilon \bm{V} _0 (\epsilon \bm{V} _0)^{\rm T}$ make up the last $(m + 2n)^2$ rows of $\bm{H}^{\epsilon} _t$: $(\epsilon \bm{V}^{\epsilon} _t)_1 \epsilon \bm{V}^{\epsilon} _t - \epsilon (\bm{V}_0)_1 \epsilon \bm{V}_0$ is the first column of $ \epsilon \bm{V}^{\epsilon} _t (\epsilon \bm{V}^{\epsilon} _t)^{\rm T} - \epsilon \bm{V} _0 (\epsilon \bm{V} _0)^{\rm T}$, $(\epsilon \bm{V}^{\epsilon} _t)_2 \epsilon \bm{V}^{\epsilon} _t - \epsilon (\bm{V}_0)_2 \epsilon \bm{V}_0$ is the second column of $ \epsilon \bm{V}^{\epsilon}_t (\epsilon \bm{V}^{\epsilon} _t)^{\rm T} - \epsilon \bm{V} _0 (\epsilon \bm{V} _0)^{\rm T}$, and so on.  Consider the expression for $ d[\epsilon \bm{V}^{\epsilon} _s (\epsilon \bm{V}^{\epsilon} _s)^{\rm T}]$ given by equation~\eqref{differential}.  Because the stochastic integrals are local martingales, the last $(m + 2n)^2$ rows of $\bm{A}^{\epsilon} _t$ are made up of the column of the Lebesgue integrals that are present in the expression for the integral of the right side of equation~\eqref{differential}:
$$ \bm{A}^{\epsilon} _t =
\left[\begin{array}{c}
t \\
\bm{0} \\
t \\
(\bm{A}^{\epsilon} _t)^1 \\
\vdots \\
(\bm{A}^{\epsilon} _t)^{m + 2n}
\end{array}\right]
$$
where
\begin{align}
\big( ( \bm{A}^{\epsilon} _t)^1,  \; & (\bm{A}^{\epsilon} _t)^2, \dots , (\bm{A}^{\epsilon} _t)^{m + 2n} \big)  = \notag \\
& \hspace{12pt} \int _0 ^t \epsilon \bm{V}^{\epsilon} _s (\bm{F}(\bm{X}^{\epsilon} _s))^{\rm T} ds + \int _0 ^t  \bm{F}(\bm{X}^{\epsilon} _s) (\epsilon \bm{V}^{\epsilon} _s)^{\rm T} ds \notag \\
& - \int _0 ^t   \epsilon \bm{V}^{\epsilon} _s (\bm{V}^{\epsilon} _s)^{\rm T} (\bm{\gamma}(\bm{X}^{\epsilon} _s))^{\rm T} ds - \int _0 ^t   \bm{\gamma}(\bm{X}^{\epsilon} _s) \bm{V}^{\epsilon} _s \epsilon (\bm{V}^{\epsilon} _s)^{\rm T}  ds \notag \\
& + \int _0 ^t \epsilon ^2 \bm{V}^{\epsilon} _s (\bm{V}^{\epsilon} _s)^{\rm T} (\bm{\kappa} (\bm{X}^{\epsilon} _s))^{\rm T} ds + \int _0 ^t \epsilon ^2 \bm{\kappa} (\bm{X}^{\epsilon} _s) \bm{V}^{\epsilon} _s (\bm{V}^{\epsilon} _s)^{\rm T}  ds + \int _0 ^t \bm{\sigma} \bm{\sigma} ^{\rm T} ds \notag
\end{align}
Thus, to show that Condition~1 holds, it suffices to show (since $\int _0 ^t \bm{\sigma} \bm{\sigma} ^{\rm T} ds$ is just a constant) that the family (indexed by $\epsilon$)
\begin{align}
& \int _0 ^t   \| \epsilon \bm{V}^{\epsilon} _s (\bm{F}(\bm{X}^{\epsilon} _s))^{\rm T} \| ds + \int _0 ^t  \| \bm{F}(\bm{X}^{\epsilon} _s) (\epsilon \bm{V}^{\epsilon} _s)^{\rm T} \| ds \notag \\
 + & \int _0 ^t  \| \epsilon \bm{V}^{\epsilon} _s  (\bm{V}^{\epsilon} _s)^{\rm T} (\bm{\gamma}(\bm{X}^{\epsilon} _s))^{\rm T} \| ds + \int _0 ^t \| \bm{\gamma}(\bm{X}^{\epsilon} _s) \bm{V}^{\epsilon} _s \epsilon (\bm{V}^{\epsilon} _s)^{\rm T} \|  ds \notag \\
 + & \int _0 ^t \| \epsilon ^2 \bm{V}^{\epsilon} _s (\bm{V}^{\epsilon} _s)^{\rm T} (\bm{\kappa} (\bm{X}^{\epsilon} _s))^{\rm T} \| ds + \int _0 ^t \| \epsilon ^2 \bm{\kappa} (\bm{X}^{\epsilon} _s) \bm{V}^{\epsilon} _s (\bm{V}^{\epsilon} _s)^{\rm T} \|  ds  \notag
\end{align}
is stochastically bounded (see the statement of Lemma 1 for the definition of a stochastically bounded family).  The first two and last two terms go to zero in probability as $\epsilon \rightarrow 0$ by Lemma~7 (note that $\bm{\kappa}$ and $\bm{F}$ are bounded by the assumptions of Theorem~1), so it suffices to show that the third and fourth terms are stochastically bounded. Since $\bm{\gamma}$ is bounded (by the assumptions of Theorem~1), it suffices to show that $E[ \| \epsilon \bm{V}^{\epsilon} _s (\bm{V}^{\epsilon} _s )^{\rm T} \|]$ is bounded uniformly in $\epsilon$.  This follows from Lemma 8 and the fact that for a vector $\bm{v}$ and outer product $\bm{vv}^{\rm T}$, $\| \bm{vv}^{\rm T} \| = \| \bm{v} \| ^2$:
\begin{equation*}
E  \left[ \| \epsilon \bm{V}^{\epsilon} _s (\bm{V}^{\epsilon} _s )^{\rm T} \| \right]  = E \left[ \epsilon \| \bm{V}^{\epsilon} _s \| ^2 \right] \leq C \; .
\end{equation*}

We now check Condition~2 of Lemma~1, where $\bm{h}^{\epsilon}$ and $\bm{h}$ are defined in equations~\eqref{h epsilon} and \eqref{h} respectively.  We first note that $\bm{J}$ is continuous and bounded given the assumption that the $g^{ij}$ are continuous and bounded (we have explicitly computed $\bm{J}$ in order to arrive at equation~\eqref{limiting equation 2}).  Part 1 of Condition~2 then follows from the boundedness of $\bm{F}$, $\bm{\kappa}$, $\bm{\gamma}$, $\frac{\partial \bm{\kappa}}{\partial X_{\ell}}$, and $\frac{\partial \bm{\gamma}}{\partial X_{\ell}}$, Lemma 2, and equation~\eqref{derivative of inverse}.  Part 2 of Condition~2 is immediate given equation~\eqref{gamma inverse} and the assumptions that the $f^i$ are continuous and the $g^{ij}$ have continuous derivatives.  This completes the proof of Theorem 1.  

\section{Discussion}

The main result of this article reduces the system of stochastic differential delay equations~\eqref{SDDE} to a simpler system (equations~\eqref{thm limiting equation} and \eqref{thm limiting equation 2}). First we use Taylor expansion to obtain the (approximate) system of SDEs~\eqref{MainSDE} and then we further simplify it by taking the limit as the time delays and correlation times of the noises go to zero.  This is useful for applications as the final equations are easier to analyze than the original ones while still being in agreement with experimental results \cite{Pesce} (see also the discussion below).

\begin{figure}[h]
\centering
\includegraphics[width=8cm]{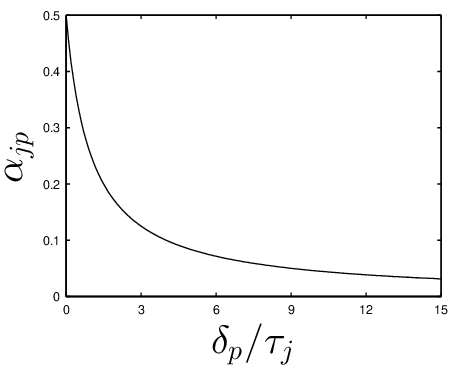}
\caption{Dependence of the coefficients $\alpha_{jp}$ of the noise-induced drift on the ratio between the corresponding delay time $\delta_p$ and noise correlation time $\tau_j$ (see equation~\eqref{alpha}). For $\delta_{p}/\tau_{j}~\rightarrow~\infty$, the solution converges to the solution of the  It\^o equation~\eqref{eq:heidodkIto}, while, for $\delta_{p}/\tau_{j}~\rightarrow~0$, it converges to the solution of the Stratonovich version~\eqref{eq:heidodk}.
}
\label{fig1}
\end{figure}

As a result of dependence of the noise coefficients on the state of the system (multiplicative noise), a \emph{noise-induced drift} appears in equation~\eqref{thm limiting equation}. It has a form analogous to that of the \emph{Stratonovich correction} to the It\^o equation with the noise term $\sum_j g^{ij} (\bm{y}_t) dW^j _t$. Each drift is a linear combination of the terms $g^{pj} (\bm{y} _t) \frac{\partial g^{ij} (\bm{y} _t)}{\partial y_p} $, but, while in the Stratonovich correction they all enter with coefficients equal to ${1\over2}$, their coefficients in the additional drift of the limiting equation~\eqref{thm limiting equation} are 
\begin{equation}
\frac{\frac{\Gamma}{\Omega ^2} \frac{\delta _p}{\tau _j} + \frac{1}{\Gamma} \left(1 - \frac{\delta _p}{\tau_j} \right)}{2 \left( \frac{\Gamma}{\Omega ^2} \frac{\delta _p}{\tau _j} \left(1 + \frac{\delta _p}{\tau _j} \right) + \frac{1}{\Gamma} \right) } \; .
\end{equation}
As noted in Remark~1, these coefficients approach their limiting values 
\begin{equation}\label{alpha}
\alpha _{jp} = {1\over2}\left(1 + \frac{\delta _p}{\tau_j}\right)^{-1},
\end{equation}
as the harmonic noise approaches the Ornstein-Uhlenbeck process, i.e. taking the limit $\Gamma, \Omega ^2 \rightarrow \infty$ while keeping $\frac{\Gamma}{\Omega ^2}$ constant (see Fig.~\ref{fig1}). One can interpret the terms of the noise-induced drift as representing different stochastic integration conventions, a point that is further explained in Ref.~\cite{Pesce}. For example,  if all $\delta_{p}/\tau_{j}~\rightarrow~\infty$, the solution converges to the solution of the It\^o equation:
\begin{equation}\label{eq:heidodkIto}
d y^i _t = f^i (\bm{y} _t) dt + \sum_j g^{ij} (\bm{y}_t)\,dW^j _t \; .
\end{equation}
On the other hand, if all $\delta_{p}/\tau_{j}~\rightarrow~0$, the solution converges to the solution of the Stratonovich version of~\eqref{eq:heidodkIto}:
\begin{equation}\label{eq:heidodk}
d y^i _t = f^i (\bm{y} _t) dt + \sum_j g^{ij} (\bm{y}_t) \circ dW^j _t \; .
\end{equation}


\begin{figure}[h!]
\centering
\includegraphics[width=13cm]{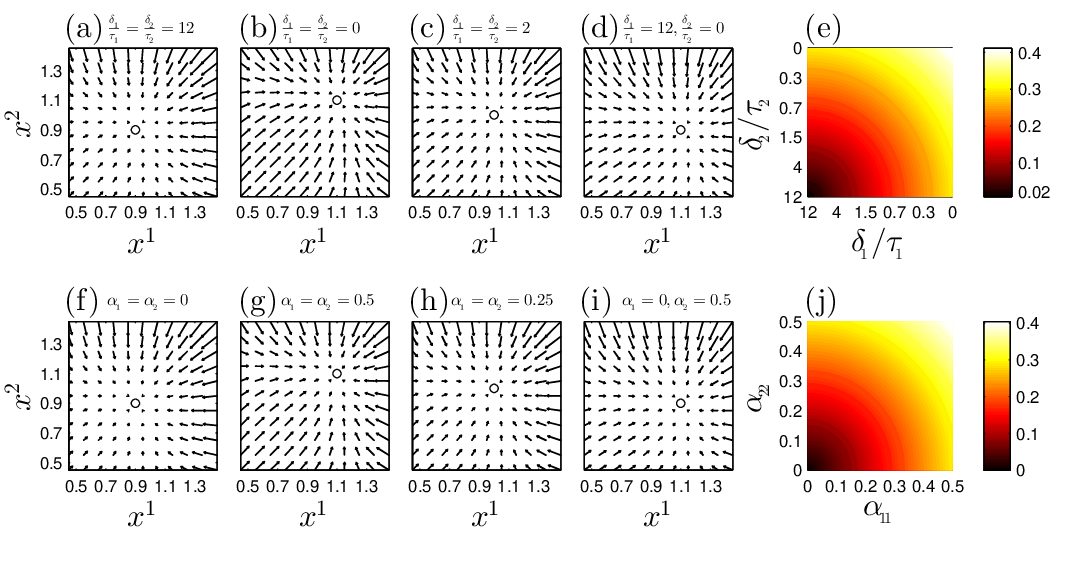}
\caption{(a-d) Drift fields (arrows) estimated from a numerical solution of the SDDEs~\eqref{model}  with colored noises ($A=B=0.1$ and $\sigma=0.2$) for various values of the ratios $\delta_1/\tau_1$ and $\delta_2/\tau_2$. The circles represent the zero-drift points. (e) Modulus of the displacement of the zero-drift point from the equilibrium position corresponding to equations~\eqref{model} without noise ($\sigma = 0$) as a function of $\delta_1/\tau_1$ and $\delta_2/\tau_2$.
(f-i) Drift fields (arrows) of the solution of the limiting SDEs~\eqref{thm limiting equation 2} corresponding to the SDDEs~\eqref{model}. $\alpha_{11}$ and $\alpha_{22}$ are given as functions of $\delta_1/\tau_1$ and $\delta_2/\tau_2$ by equation~\eqref{alpha}. The circles represent the zero-drift points. There is good agreement between (f-i) and (a-d). (j) Modulus of the displacement of the zero-drift point from the equilibrium position corresponding to equations~\eqref{model} without noise ($\sigma = 0$) for the solution of the limiting SDEs~\eqref{thm limiting equation 2} corresponding to the SDDEs~\eqref{model} as a function of $\alpha_{11}$ and $\alpha_{22}$. Again, (j) and (e) are in good agreement.}
\label{fig2}
\end{figure}

While convergence of equations~\eqref{MainSDE} to \eqref{thm limiting equation 2} is rigorously proven in this article, a specific system with non-zero values of $\delta_p$ and $\tau_j$ is more accurately described by \eqref{MainSDE} than by \eqref{thm limiting equation 2}. In addition, equations~\eqref{MainSDE} were obtained from the original system \eqref{SDDE} by an approximation (Taylor expansion). It is thus important to compare the behavior of the numerical solutions of \eqref{SDDE} and \eqref{thm limiting equation 2} in a particular case. As an example, we consider the two-dimensional system
\begin{equation}\label{model}
\left\{\begin{array}{ccc}
dx_{t}^{1} & = & A\, x_{t}^{1}\, (1-x_{t}^{1}-B\, x_{t}^{2})\, dt + \sigma\, x_{t-\delta_1}^{1}\, \eta_{t}^{1}\, dt \\
dx_{t}^{2} & = & A\, x_{t}^{2}\, (1-x_{t}^{2}-B\, x_{t}^{1})\, dt + \sigma\, x_{t-\delta_2}^{2}\, \eta_{t}^{2}\, dt \\
\end{array}\right.
\end{equation}
where $A,\,B,$ and $ \sigma$ are non-negative constants, $\eta_{t}^{1}$ and $\eta_{t}^{2}$ are colored noises with correlation times $\tau_1$ and $\tau_2$ respectively, and $\delta_1$ and $\delta_2$ are the delay times. These equations can describe, e.g., the dynamics of a noisy ecosystem where two populations are present whose sizes are proportional to the state variables $x_1$ and $x_2$. In the absence of noise ($\sigma=0$) the system described by equations~\eqref{model} is known as the competitive Lotka-Volterra model \cite{Murray} and has only one stable fixed point for which $x_{\rm eq}^{1}, x_{\rm eq}^{2} \neq 0$ at $x_{\rm eq}^{1} = x_{\rm eq}^{2} = (1+B)^{-1}$. For a noisy system (with or without delay), fixed points of the corresponding deterministic flow that is generated by the drift no longer describe equilibria. One can still estimate the system's drift field, as done in Ref.~\cite[Methods]{Pesce}, and identify the points in the state space where the drift is zero. For the system described by equations~\eqref{model}, the drift fields and the coordinates of the zero-drift point (for which $x
^{1}, x^{2} \neq 0$) depend on $\delta_1/\tau_1$ and $\delta_2/\tau_2$, as shown in Figs.~\ref{fig2}(a-e) for $A=B=0.1$ and $\sigma=0.2$. We now calculate the drift fields and zero-drift points of the corresponding limiting SDEs~\eqref{thm limiting equation 2}. The results, shown in Figs.~\ref{fig2}(f-j), are in good agreement with the ones obtained by directly simulating equation~\eqref{model}.


\section*{Acknowledgements}
We would like to thank the referee of this paper for insightful comments and the referee of the earlier \emph{Nature Communications} paper~\cite{Pesce} who suggested that we 
consider the main equation in multiple dimensions and with different time delays.  This suggestion led to the more general result presented here that more clearly reveals the interplay between the time delays and the correlation times of the noises. A.M. and J.W. were partially supported by the NSF grants DMS 1009508 and DMS 0623941. G.V. was partially supported by the Marie Curie Career Integration Grant (MC-CIG) No. PCIG11 GA-2012-321726.

\appendix 

\section*{Appendix}
We list some facts about the harmonic noise process.  The stationary harmonic noise process, defined as the stationary solution to \eqref{eq:harmonicnoise}, satisfies \cite{Schimansky-Geier, Wang}
\begin{equation}
\label{propertiesHNstat}
E[\eta ^j _t] = E[z ^j _t] = 0 \; , \hspace{15pt} E [ (\eta^j _t) ^2] = \frac{1}{2 \tau _j} \; , \hspace{15pt} E[(z ^j _t) ^2] = \frac{\Omega ^2}{2 \tau _j} \; ,
\end{equation}
and has covariance function
\begin{equation}
E[\eta ^j _t \eta ^j _{t + s}] = \frac{1}{2 \tau _j} e^{- \frac{\Gamma ^2}{2 \Omega ^2 \tau _j} s} \left[ \cos(\omega _1 s) + \frac{\Gamma ^2}{2 \tau _j \Omega ^2 \omega _1} \sin(\omega_1 s) \right], \hspace{5pt} s \geq 0
\end{equation}
where
$$ \omega_1 = \frac{\Gamma}{\Omega \tau _j} \sqrt{1 - \frac{\Gamma ^2}{4 \Omega ^2}}$$
We state a result concerning the convergence of the harmonic noise process to an Ornstein-Uhlenbeck process as $\Gamma, \Omega ^2 \rightarrow \infty$ while the ratio $\frac{\Gamma}{\Omega ^2}$ remains constant.  Letting $\tilde{\eta} ^j _t = \tau _j  \frac{\Omega ^2}{\Gamma} \eta ^j _t$, equation~\eqref{eq:harmonicnoise} becomes
$$
\begin{array}{ccl}
d \tilde{\eta}_t^j & = &  z ^j _t dt \\
dz ^j _t & = & - \frac{1}{\tau _j} \frac{\Gamma}{\Omega ^2} \Gamma z ^j _t dt - \frac{1}{\tau _j ^2} \frac{\Gamma}{\Omega ^2} \Gamma \tilde{\eta} ^j _t dt + \frac{1}{\tau _j}\Gamma dW ^j _t \; .
\end{array}
$$
Note that this is a system of linear SDEs with constant coefficients, and so it can be solved explicitly.  Thus, its limit can be studied directly, and we have the following result (this result can also be shown using the theorem of Hottovy et al. \cite{Hottovy}).  Let $\tilde{\chi} _t ^j$ be the solution to
\begin{equation*}
 d \tilde{\chi} _t ^j = - \frac{1}{\tau _j} \tilde{\chi} _t ^j dt + \frac{\Omega ^2}{\Gamma} d W _t ^j \; .
 \end{equation*} 
Then, as $\Gamma, \Omega ^2 \rightarrow \infty$ while the ratio $\frac{\Gamma}{\Omega ^2}$ remains constant,  $\tilde{\eta} _t^j$ converges to $\tilde{\chi} _t ^j$ in $L^2$ with respect to $C ([0, T], \mathbb{R})$, that is,
$$ \lim _{\Gamma \rightarrow \infty \;  (\frac{\Gamma}{\Omega ^2} \; \mathrm{constant})} E \left[ \left( \sup_{0 \leq t \leq T} |\tilde{\eta} _t ^j - \tilde{\chi} _t ^j| \right) ^2 \right] = 0 \; .
$$
Thus, letting $\chi _t ^j$ be the solution to
\begin{equation*}
 d \chi _t ^j = - \frac{1}{\tau _j} \chi _t ^j dt + \frac{1}{\tau _j} d W _t ^j
\end{equation*}
so that $\chi _t ^j$ is an Ornstein-Uhlenbeck process with correlation time $\tau _j$, we have that as $\Gamma, \Omega ^2 \rightarrow \infty$ while the ratio $\frac{\Gamma}{\Omega ^2}$ remains constant, $\eta _t ^j$ converges to $\chi _t ^j$ in $L^2$ (and therefore in probability) with respect to $C ([0, T], \mathbb{R}) $.

\end{document}